\documentclass[a4]{amsart}

\input xypic
\usepackage{amssymb} 

\newtheorem{theorem}{Theorem}[section]
\newtheorem{proposition}[theorem]{Proposition}
\newtheorem{lemma}[theorem]{Lemma}
\newtheorem{corollary}[theorem]{Corollary}

\theoremstyle{definition}

\newtheorem{remark}[theorem]{Remark}

\newcommand{\oddanick}{\ensuremath{T^{2n+1}(p^{r})}} 
 
\newcommand{\anick}{\ensuremath{T^{2n+1}(2^{r})}} 
\newcommand{\anickplus}{\ensuremath{T^{2n+1}(2^{r+1})}} 
\newcommand{\curly}{\ensuremath{S^{2n+1}\{\underline{2}^{r}\}}} 
\newcommand{\curlyplus}{\ensuremath{S^{2n+1}\{\underline{2}^{r+1}\}}} 
\newcommand{\jkm}{\ensuremath{J_{k-1}(S^{2n})}}
\newcommand{\jk}{\ensuremath{J_{k}(S^{2n})}} 
\newcommand{\moore}{\ensuremath{P^{2n+1}(2^{r})}}

\newcommand{\hlgy}[1]{\ensuremath{H_{*}(#1)}} 
\newcommand{\rhlgy}[1]{\ensuremath{\widetilde{H}_{*}(#1)}} 

\newcounter{bean} 
\newenvironment{letterlist}{\begin{list}{\rm ({\alph{bean}})} 
      {\usecounter{bean}\setlength{\rightmargin}{\leftmargin}}}
      {\end{list}}  

\newcommand{\namedright}[3]{\ensuremath{#1\stackrel{#2}
 {\longrightarrow}#3}}
\newcommand{\nameddright}[5]{\ensuremath{#1\stackrel{#2}
 {\longrightarrow}#3\stackrel{#4}{\longrightarrow}#5}} 
\newcommand{\namedddright}[7]{\ensuremath{#1\stackrel{#2}
 {\longrightarrow}#3\stackrel{#4}{\longrightarrow}#5
  \stackrel{#6}{\longrightarrow}#7}} 

\newcommand{\larrow}{\relbar\!\!\relbar\!\!\rightarrow} 
\newcommand{\llarrow}{\relbar\!\!\relbar\!\!\larrow} 
\newcommand{\lllarrow}{\relbar\!\!\relbar\!\!\llarrow} 

\newcommand{\lnamedright}[3]{\ensuremath{#1\stackrel{#2}
 {\larrow}#3}}
\newcommand{\lnameddright}[5]{\ensuremath{#1\stackrel{#2}
 {\larrow}#3\stackrel{#4}{\larrow}#5}}

\newcommand{\llnameddright}[5]{\ensuremath{#1\stackrel{#2}
 {\llarrow}#3\stackrel{#4}{\llarrow}#5}} 
 
\newcommand{\lllnamedright}[3]{\ensuremath{#1\stackrel{#2}
 {\lllarrow}#3}}
\newcommand{\lllnameddright}[5]{\ensuremath{#1\stackrel{#2}
 {\lllarrow}#3\stackrel{#4}{\lllarrow}#5}}

\newcommand{\qqed}{\hfill\Box}

\newcommand{\zmodtwo}{\ensuremath{\mathbf{Z}/2\mathbf{Z}}}

\begin{document} 


\title{$2$-Primary Anick fibrations} 
\author{Stephen D. Theriault}
\address{Department of Mathematical Sciences, University
         of Aberdeen, Aberdeen AB24 3UE, United Kingdom}
\email{s.theriault@maths.abdn.ac.uk}



\begin{abstract} 
Cohen conjectured that for $r\geq 2$ there is a space \anick\ 
and a homotopy fibration sequence 
\(\namedddright{\Omega^{2} S^{2n+1}}{\varphi_{r}}{S^{2n-1}}{} 
    {\anick}{}{\Omega S^{2n+1}}\) 
with the property that the composition 
\(\nameddright{\Omega^{2} S^{2n+1}}{\varphi_{r}}{S^{2n-1}}{E^{2}} 
     {\Omega^{2} S^{2n+1}}\) 
is homotopic to the $2^{r}$-power map. We positively resolve 
this conjecture when $r\geq 3$. Several preliminary results are 
also proved which are of interest in their own right. 
\end{abstract} 

\maketitle

\section{Introduction} 
\label{sec:intro} 

Let $p$ be a prime and $r$ a non-negative integer. Localize 
spaces and maps at $p$. Let 
\(E^{2}\colon\namedright{S^{2n-1}}{}{\Omega^{2} S^{2n+1}}\) 
be the double suspension. When $p\geq 5$ and $r\geq 1$,  
Anick~\cite{A} proved the existence of a space \oddanick\ and a 
homotopy fibration sequence 
\begin{equation} 
  \label{anickfib} 
  \namedddright{\Omega^{2} S^{2n+1}}{\varphi_{r}}{S^{2n-1}}{} 
    {\oddanick}{}{\Omega S^{2n+1}}  
\end{equation} 
with the property that $E^{2}\circ\varphi_{r}$ is homotopic to 
the $p^{r}$-power map on $\Omega^{2} S^{2n+1}$. The proof of this 
was long and complex. A much simpler proof was given  
in~\cite{GT} for $p\geq 3$ and $r\geq 1$ by using a new method 
for constructing certain extensions. A variation of this  
method was used in~\cite{T} to reproduce Cohen, Moore, and 
Neisendorfer's result~\cite{CMN1,CMN2} that $p^{n}$ annihilates  
the $p$-torsion in $\pi_{\ast}(S^{2n+1})$ when $p$ is odd. 

The new method raised the possibility of producing $2$-primary 
Anick fibrations. Such fibrations were conjectured to exist by 
Cohen if $r\geq 2$. When $r=1$ the fibrations cannot exist for 
reasons related to the non-existence of elements of Hopf invariant 
one. In this paper, we positively resolve Cohen's conjecture 
when $r\geq 3$. The approach is a modification of that in~\cite{T}. 
Specifically, we prove the following. 
 
\begin{theorem} 
   \label{main} 
   Let $r\geq 3$. There is a space \anick\ and a homotopy fibration 
   sequence 
   \[\namedddright{\Omega^{2} S^{2n+1}}{\varphi_{r}}{S^{2n-1}}{} 
     {\anick}{}{\Omega S^{2n+1}}\] 
   with the property that the composition 
   \(\nameddright{\Omega^{2} S^{2n+1}}{\varphi_{r}}{S^{2n-1}}{E^{2}} 
        {\Omega^{2} S^{2n+1}}\) 
   is homotopic to the $2^{r}$-power map. 
\end{theorem} 

The $r=2$ case remains open; it falls outside the present approach because 
we proceed in two steps, each of which involves introducing an extra 
power of $2$ to annihilate certain obstructions. 

In order to apply the approach in~\cite{T}, two preliminary 
results need to be proved, each of which is somewhat involved. 
The first is given in Theorem~\ref{improvedR}; to state it requires 
requires several definitions. The $EHP$-sequence is a homotopy 
fibration sequence
\(\namedddright{\Omega^{2} S^{2n+1}}{P}{S^{n}}{E}{\Omega S^{n+1}}
     {H}{\Omega S^{2n+1}}\)
where $H$ is the second James-Hopf invariant, $E$ is the
suspension, and $P$ is the connecting map for the fibration.
This fibration forms the basis for the $EHP$-spectral sequence
which calculates the homotopy groups of spheres. Let
\(\underline{2}:\namedright{S^{2n+1}}{}{S^{2n+1}}\)
be the map of degree~$2$, and let
\(2:\namedright{\Omega S^{2n+1}}{}{\Omega S^{2n+1}}\)
be the second-power map. Let $\Omega S^{2n+1}\{2\}$ be the
homotopy fiber of the second-power map, and let $D^{2n+1}$ be the
homotopy fiber of the difference $\Omega\underline{2}-2$.
In~\cite[\S 4]{C} it is shown that the map
\(\namedright{\Omega^{2} S^{2n+1}}{\Omega H}{\Omega^{2} S^{4n+1}}\)
has order~$2$, implying that $\Omega H$ lifts to a map
\(S:\namedright{\Omega^{2} S^{2n+1}}{}{\Omega^{2} S^{4n+1}\{2\}}\).
The same argument in~\cite[\S 4]{C} shows that the composite
\(\llnameddright{\Omega^{2} S^{2n}}{\Omega H}{\Omega^{2} S^{4n-1}}
    {\Omega^{2}\underline{2}-2}{\Omega^{2} S^{4n-1}}\)
is null homotopic, implying that in this case $\Omega H$ lifts to
a map
\(S^{\prime}:\namedright{\Omega^{2} S^{2n}}{}{\Omega D^{4n-1}}\).
We prove the following.

\begin{theorem}
   \label{improvedR}
   There are choices of the lifts $S$ and $S^{\prime}$ for which
   there are homotopy commutative diagrams
   \[\diagram
        \Omega S^{2n}\rto^-{h}\dto^{\Omega E}
            & \Omega S^{4n-1}\rto^-{\Omega E^{2}}
            & \Omega^{3} S^{4n+1}\dto
        & & \Omega S^{2n-1}\rto^-{h}\dto^{\Omega E}
            & \Omega S^{4n-3}\rto^-{\Omega E^{2}}
            & \Omega^{3} S^{4n-1}\dto \\
         \Omega^{2} S^{2n+1}\rrto^-{S}
            & & \Omega^{2} S^{4n+1}\{2\}
        & & \Omega^{2} S^{2n}\rrto^-{S^{\prime}}
            & & \Omega D^{4n-1}
     \enddiagram\]
   where the maps labelled $h$ are Hopf invariants in the sense
   that there in each case there is a homotopy fibration
   \(\nameddright{S^{n}}{E}{\Omega S^{n+1}}{h}{\Omega S^{2n+1}}\).
\end{theorem}

Taking homotopy fibers vertically in the two diagrams in 
Theorem~\ref{improvedR} we obtain the following.

\begin{corollary}
   \label{Richter}
   There are homotopy commutative diagrams
   \[\diagram
         \Omega^{3} S^{4n+1}\rrdouble\dto^{P}
            & & \Omega^{3} S^{4n+1}\dto^{2}
         & & \Omega^{3} S^{4n-1}\rrdouble\dto^{P}
            & & \Omega^{3} S^{4n-1}\dto^{\Omega^{3}\underline{2}-2} \\
         \Omega S^{2n}\rto^-{h} & \Omega S^{4n-1}\rto^-{\Omega E^{2}}
            & \Omega^{3} S^{4n+1}
         & & \Omega S^{2n-1}\rto^-{h}
            & \Omega S^{4n-3}\rto^-{\Omega E^{2}}
            & \Omega^{3} S^{4n-1}.
     \enddiagram\]
   $\qqed$
\end{corollary} 

Corollary~\ref{Richter} was proved by Richter~\cite{R}, who also showed 
that the map $h$ could be taken to be the second James-Hopf invariant $H$. 
Theorem~\ref{improvedR} can therefore be viewed as showing that there is 
an advantageous choice of the lift $S$ which lets us ``deloop" Richter's 
result.  

An earlier version of the left diagram in Corollary~\ref{Richter}
was proved by Selick~\cite{S} which involved a factorization of the
$2^{nd}$-power map on $\Omega^{4} S^{2n+1}$ through $\Omega^{2} E^{2}$.
This was used to show that $2^{\frac{3}{2}n+\epsilon}$ annihilates
the $2$-torsion in $\pi_{\ast}(S^{2n+1})$, where $\epsilon=0$ if $n$
is even and $\epsilon=1$ if $n$ is odd. This is the best
known upper bound on the $2$-primary homotopy exponent of spheres. The
right diagram in Corollary~\ref{Richter} is related to the very
delicate problem of determining whether the difference
$d=\Omega\underline{2}-2$ on $\Omega S^{2n+1}$ is null homotopic
after some number of loops, which would imply that the degree $2$
map induces multiplication by $2$ on homotopy groups. It is
known~\cite[\S 11]{C} that $d$ is null homotopic if and only if
$n\in\{0,1,3\}$, which are cases related to the existence of an element
of Hopf invariant one, and $\Omega d$ is null homotopic if and
only if there exists an element of Arf invariant one
in $\pi_{4n+1}(S^{2n+1})$. On the other hand, in~\cite{CJ} it
is shown that for certain values of $n$ the difference $\Omega^{k} d$
is essential for large values of $k$. Possibly properties of the
map $S^{\prime}$ in Theorem~\ref{improvedR} will help give more
information about $\Omega^{2} d$. 

The second preliminary result concerns the fiber $W_{n}$ of the 
double suspension 
\(\namedright{S^{2n-1}}{E^{2}}{\Omega^{2} S^{2n+1}}\). 
In what follows we will need to know that the $4^{th}$-power map 
on $\Omega W_{2n}$ is null homotopic. We show this is the case 
for all $\Omega W_{n}$. This uses a factorization of the $4^{th}$-power 
map on $\Omega^{2} S^{2n+1}$ through $E^{2}$ which is of interest 
in its own right. 

\begin{theorem}
   \label{4factor}
   There is a homotopy commutative diagram
   \[\diagram
         \Omega^{2} S^{2n+1}\rto\dto^{4} & S^{2n-1}\dto^{E^{2}} \\
         \Omega^{2} S^{2n+1}\rdouble & \Omega^{2} S^{2n+1}.
     \enddiagram\]
\end{theorem}

A weaker version of Theorem~\ref{4factor} with one more loop was proved 
in~\cite[\S 5]{C}. 

In~\cite{B-R} it was shown that the $4^{th}$-power
map on $\Omega^{2} W_{n}$ is null homotopic, implying that
$\mbox{$4\cdot\pi_{\ast}(W_{n})=0$}$. This is best possible in the sense
that the $2^{nd}$-power map on $\Omega^{2} W_{n}$ is known to
be essential. We improve this by removing one loop.

\begin{theorem}
   \label{Wnpower}
   The $4^{th}$-power map on $\Omega W_{n}$ is null homotopic.
\end{theorem}

Since $W_{n}$ is homotopy equivalent to $\Omega BW_{n}$, it is
an $H$-space. It would be interesting to know whether the
$4^{th}$-power map on $W_{n}$ is null homotopic.
\medskip 

This paper is organized as follows. The first preliminary result, 
Theorem~\ref{improvedR}, is proved in Sections~\ref{sec:reduceddiag} 
and~\ref{sec:projplane}. The precursor to the second preliminary result,  
Theorem~\ref{4factor}, is proved in Section~\ref{sec:4factor}, while  
the second preliminary result itself, Theorem~\ref{Wnpower}, is proved 
in Section~\ref{sec:expWn}. With these in hand, Theorem~\ref{main} 
is outlined in Section~\ref{sec:proof} with some details deferred to 
Section~\ref{sec:deferredproofs}.

\section{The fiber of the reduced diagonal}
\label{sec:reduceddiag}

The context in which Theorem~\ref{improvedR} is proved involves
the fiber of the reduced diagonal, applied naturally to the map
\(\namedright{S^{n}}{E}{\Omega S^{n+1}}\).
Let
\(\Delta:\namedright{X}{}{X\times X}\)
be the diagonal map. Composing with the quotient map
\(\namedright{X\times X}{}{X\wedge X}\)
gives the reduced diagonal
\(\overline{\Delta}:\namedright{X}{}{X\wedge X}\).
The map $\overline{\Delta}$ is natural for maps
\(\namedright{X}{}{Y}\)
and induces a natural homotopy fibration sequence
\[\namedddright{\Omega (X\wedge X)}{}{F(X)}{}{X}{\overline{\Delta}}
    {X\wedge X}.\]
Here, the fiber $F(X)$ is defined specifically as the topological
pullback of $\overline{\Delta}$ and the evaluation map
\(\namedright{P(X\wedge X)}{}{X\wedge X}\),
where $P(X\wedge X)$ is the path space on $X\wedge X$.

In the special case when $X=\Omega S^{n+1}$, Richter~\cite{R} showed
that there is a homotopy commutative diagram
\begin{equation}
  \label{Richterdiag}
  \diagram
      \Omega S^{n+1}\rto^-{H}\dto^{\overline{\Delta}}
          & \Omega S^{2n+1}\dto^{1+\Omega(-\underline{1}^{n})} \\
      \Omega S^{n+1}\wedge\Omega S^{n+1}\rto^-{\alpha}
          & \Omega S^{2n+1}
  \enddiagram
\end{equation}
for some map $\alpha$. An explicit choice of the map $\alpha$ was
described in~\cite{R}, but we will not need this. We only care that
there is some choice of $\alpha$ that makes the diagram homotopy
commute. Let $G^{2n+1}$ be the homotopy fiber of
$1+\Omega(-\underline{1}^{n})$. Observe that when $n=2m$ then
$1+\Omega(-\underline{1}^{n})=2$ and so $G^{4m+1}\simeq\Omega
S^{4m+1}\{2\}$. When $n=2m-1$ then
$1+\Omega(-\underline{1}^{n})=1+\Omega(-\underline{1})$ 
and by~\cite[\S 4]{C} the latter is homotopic to 
$\Omega\underline{2}-2$. Thus $G^{4m-1}$ is homotopy equivalent 
to the space $D^{4m-1}$ defined in the Introduction. Taking
vertical homotopy fibers in~(\ref{Richterdiag}) gives a homotopy
commutative diagram
\begin{equation}
  \label{Richterdiag2}
  \diagram
      F(\Omega S^{n+1})\rto^-{f}\dto & G^{2n+1}\dto \\
      \Omega S^{n+1}\rto^-{H} & \Omega S^{2n+1}
  \enddiagram
\end{equation}
for some map $f$ of fibers.

Now consider precomposing with the suspension
\(\namedright{S^{n}}{E}{\Omega S^{n+1}}\).
The naturality of the reduced diagonal gives a homotopy commutative
diagram
\[\diagram
     F(S^{n})\rto^-{F(E)}\dto & F(\Omega S^{n+1})\dto \\
     S^{2n}\rto^-{E} & \Omega S^{2n+1}.
  \enddiagram\]
Juxtapose this diagram with~(\ref{Richterdiag2}). Since $H\circ E$ is
null homotopic, so is the composite
\(\namedddright{F(S^{n})}{F(E)}{F(\Omega S^{2n+1})}{f}{G^{2n+1}}
   {}{\Omega S^{2n+1}}\).
Thus $f\circ F(E)$ lifts to the homotopy fiber of
\(\namedright{G^{2n+1}}{}{\Omega S^{2n+1}}\),
giving a homotopy commutative diagram
\begin{equation}
  \label{Fdgrm}
  \diagram
     F(S^{n})\rto^-{g}\dto^{F(E)} & \Omega^{2} S^{2n+1}\dto \\
     F(\Omega S^{2n+1})\rto^-{f} & G^{2n+1}
  \enddiagram
\end{equation}
for some map $g$. We wish to analyze~(\ref{Fdgrm}) after looping.

Since $S^{n}$ is a co-$H$ space, its reduced diagonal is null
homotopic, and so $F(S^{n})\simeq S^{n}\times\Omega(S^{n}\wedge S^{n})$.
Let $g_{1}$ and $g_{2}$ be the restrictions of
\(\namedright{F(S^{2n})}{g}{\Omega^{2} S^{2n+1}}\)
to $S^{n}$ and $\Omega(S^{n}\wedge S^{n})$ respectively. After looping,
multiplicativity implies that $\Omega g$ is determined by the
restrictions $\Omega g_{1}$ and $\Omega g_{2}$. Since $G^{2n+1}$ is
$(2n-2)$-connected, $g_{1}$ is null homotopic. Thus $\Omega g$ factors
through $\Omega g_{2}$, which we state as follows.

\begin{lemma}
   \label{omegagfact}
   There is a homotopy commutative diagram
   \[\diagram
        \Omega F(S^{n})\rto^-{\Omega\pi_{2}}\ddouble
           & \Omega^{2}(S^{n}\wedge S^{n})\dto^{\Omega g_{2}} \\
        \Omega F(S^{n})\rto^-{\Omega g} & \Omega^{3} S^{2n+1}.
     \enddiagram\]
   where $\pi_{2}$ is the projection.~$\qqed$
\end{lemma}

\section{Projective planes and the proof of Theorem~\ref{improvedR}}
\label{sec:projplane}

We begin with two general results in Theorems~\ref{Milgram}
and~\ref{Harper}. The first is due to Milgram~\cite{M}. To state 
it we first need to define several spaces and maps. Let $X$ be a 
simply-connected space. Let
\(\mu^{\ast}:\namedright{\Sigma\Omega X\wedge\Omega X}{}
  {\Sigma\Omega X}\)
be the canonical Hopf construction. This induces a homotopy cofibration
\(\namedddright{\Sigma\Omega X\wedge\Omega X}{\mu^{\ast}}
    {\Sigma\Omega X}{i}{PP(\Omega X)}{\rho}
    {\Sigma^{2}\Omega X\wedge\Omega X}\)
which defines the space $PP(\Omega X)$ and the maps $i$ and $\rho$.
The space $PP(\Omega X)$ is called the projective plane of $\Omega X$.  
To analyze $PP(\Omega X)$ more closely Milgram gave an explicit 
point-set model. Using this he was able to prove a factorization 
of the reduced diagonal  
\(\namedright{PP(\Omega X)}{\overline{\Delta}}  
   {PP(\Omega X)\wedge PP(\Omega X)}\), 
which is stated in Theorem~\ref{Milgram}. The model comes 
equipped with a canonical extension of the evaluation map 
\(ev:\namedright{\Sigma\Omega X}{}{X}\)
to a map
\(\overline{ev}:\namedright{PP(\Omega X)}{}{X}\). 
Let 
\(\tau:\namedright{\Sigma A\wedge B}{}{A\wedge\Sigma B}\) 
be the map which switches the suspension coordinate. 

\begin{theorem} 
   \label{Milgram} 
   There is a factorization 
   \[\diagram 
        PP(\Omega X)\rto^{\rho}\dto^{\overline{\Delta}} 
           & \Sigma^{2}\Omega X\wedge\Omega X\rto^-{\Sigma\tau} 
           & \Sigma\Omega X\wedge\Sigma\Omega X\dto^{i\wedge i} \\ 
        PP(\Omega X)\wedge PP(\Omega X)\rrdouble 
           & & PP(\Omega X)\wedge PP(\Omega X). 
     \enddiagram\] 
   Moreover, this factorization is natural for maps 
   \(\namedright{X}{}{X^{\prime}}\). 
   $\qqed$ 
\end{theorem} 

The naturality of the reduced diagonal gives a commutative diagram 
\[\diagram 
    PP(\Omega X)\rto^-{\overline{\Delta}}\dto^{\overline{ev}} 
       & PP(\Omega X)\wedge PP(\Omega X)
            \dto^{\overline{ev}\wedge\overline{ev}} \\ 
    X\rto^-{\overline{\Delta}} & X\wedge X. 
  \enddiagram\] 
Combining this with the commutative diagram in Theorem~\ref{Milgram} 
and using the fact that $\overline{ev}\circ i\simeq ev$ immediately 
gives the following. 

\begin{corollary}
   \label{Milgramcor}
   There is a homotopy commutative diagram
   \[\diagram
        PP(\Omega X)\rto^-{\rho}\dto^{\overline{ev}}
           & \Sigma^{2}\Omega X\wedge\Omega X\rto^-{\Sigma\tau}
           & \Sigma\Omega X\wedge\Sigma\Omega X\dto^{ev\wedge ev} \\
        X\rrto^-{\overline{\Delta}} & & X\wedge X.
     \enddiagram\]
   Moreover, this diagram is natural for maps 
   \(\namedright{X}{}{X^{\prime}}\). 
   $\qqed$ 
\end{corollary} 

Theorem~\ref{Milgram} was referred to by Richter~\cite{R} as
the basis of an alternative proof of his results, which we have
stated in Corollary~\ref{Richter}. No details were given, and it
may or may not be the case that the alternative proof went along
the lines presented here.

The second general result concerns Peterson and Stein's~\cite{PS} 
formulas for secondary cohomology operations, which was reformulated 
in terms of spaces by Harper~\cite[Ch. 3]{H}. Suppose there is a 
sequence of maps 
\(\namedddright{A}{s}{B}{}{E}{t}{D}\). 
Let $C$ and $M$ be the mapping cones of $s$ and $t$ 
respectively. Let $F$ be the topological fiber of $t$, that is, 
the pullback of $t$ and the path-evaluation map
\(\namedright{P(D)}{}{D}\).
Suppose there is a homotopy 
\(H:\namedright{B\times I}{}{D}\) 
between the composites 
\(\nameddright{B}{}{E}{t}{D}\) 
and 
\(\nameddright{B}{}{C}{}{D}\). 
Then the following holds.

\begin{theorem}
   \label{Harper}
   There is a homotopy commutative diagram
   \[\diagram
        A\rto\dto^{a} & B\rto\dto & C\rto\dto & \Sigma A\dto^{c} \\
        F\rto\dto^{b} & E\rto & D\rto & M \\
        \Omega M & & &
     \enddiagram\]
   where the map $c$ is homotopic to the adjoint of $b\circ a$.~$\qqed$
\end{theorem}

\begin{remark}
\label{Harperremark} 
The construction in~\cite[Ch. 3]{H} is natural for a strictly commuting 
map of sequences 
\[\diagram 
    A\rto\dto & B\rto\dto^{b} & E\rto\dto & D\dto^{d} \\ 
    A^{\prime}\rto & B^{\prime}\rto & C^{\prime}\rto & D^{\prime} 
  \enddiagram\] 
which is compatible with the homotopies  
\(H:\namedright{B\times I}{}{D}\) 
and 
\(H^{\prime}:\namedright{B^{\prime}\times I}{}{D^{\prime}}\), 
that is, such that $d\circ H\simeq H^{\prime}\circ (b\times 1)$.  
\end{remark}

In our case, apply Theorem~\ref{Harper} to the homotopy
commutative diagram in Corollary~\ref{Milgramcor} to obtain a homotopy 
commutative diagram
\begin{equation}
  \label{PSdgrm}
  \diagram
     \Sigma\Omega X\rto\dto^{a}
         & PP(\Omega X)\rto^-{\partial}\dto^{\overline{ev}}
         & \Sigma^{2}\Omega X\wedge\Omega X
             \rrto^-{\Sigma\mu^{\ast}}\dto^{(ev\wedge ev)\circ\Sigma\tau}
         & & \Sigma^{2}\Omega X\dto^{c} \\
     F(X)\rto\dto^{b} & X\rto^-{\overline{\Delta}}
         & X\wedge X\rrto & & C(X) \\
     \Omega C(X) & & &
  \enddiagram
\end{equation}
where $C(X)$ is the mapping cone of $\overline{\Delta}$, $F(X)$ is
the topological fiber of $\overline{\Delta}$, and $c$ is homotopic to 
the adjoint of~$b\circ a$. Further, the naturality statements in 
Corollary~\ref{Milgramcor} and Remark~\ref{Harperremark} imply that 
this diagram is natural for maps 
\(\namedright{X}{}{X^{\prime}}\). 

We will use~(\ref{PSdgrm}) in two ways. First, the naturality 
property applied to the map
\(\namedright{S^{n}}{E}{\Omega S^{n+1}}\)
gives, in particular, a homotopy commutative diagram
\begin{equation}
  \label{sigmaomegadgrm}
  \diagram
       \Sigma\Omega S^{n}\rto^-{a}\dto^{\Sigma\Omega E}
         & F(S^{2n})\dto^{F(E)} \\
       \Sigma\Omega^{2} S^{n+1}\rto^-{a} & F(\Omega S^{n+1}).
    \enddiagram
\end{equation} 
Combining this with material in Section~\ref{sec:reduceddiag}  
gives a homotopy commutative diagram
\begin{equation}
  \label{fulldgrm}
  \diagram
     & & & \Omega^{2}(S^{n}\wedge S^{n})\dto^{\Omega g_{2}} \\
     \Omega S^{n}\rto^{E}\dto^{\Omega E}
         & \Omega\Sigma\Omega S^{n}
              \rto^-{\Omega a}\dto^{\Omega\Sigma\Omega E}
         & \Omega F(S^{n})\urto^-{\Omega\pi_{2}}
             \rto^-{\Omega g}\dto^{\Omega F(E)}
         & \Omega^{3} S^{2n+1}\dto \\
     \Omega^{2} S^{n+1}\rto^-{E}
         & \Omega\Sigma\Omega^{2} S^{n+1}\rto^-{\Omega a}
         & \Omega F(\Omega S^{n+1})\rto^-{\Omega f}
         & \Omega G.
  \enddiagram
\end{equation}
Here, the left square homotopy commutes by the naturality of the
suspension, the middle square is due to~(\ref{sigmaomegadgrm}),
the right square to~(\ref{Fdgrm}), and the upper triangle to
Lemma~\ref{omegagfact}.

The second way in which we use~(\ref{PSdgrm}) is show that the
composite $\Omega g_{2}\circ\Omega\pi_{2}\circ\Omega a\circ E$
along the top in~(\ref{fulldgrm}) is homotopic to a composite
\(\nameddright{\Omega S^{n}}{h}{\Omega S^{2n-1}}{\Omega E^{2}}
     {\Omega^{3} S^{2n+1}}\)
for a map $h$ which is an epimorphism in homology. Doing so 
will let us conclude Theorem~\ref{improvedR}.

We begin with some general observations. The adjoint of a map
\(s:\namedright{\Sigma A}{}{B}\)
is homotopic to the composite
\(\nameddright{A}{E}{\Omega\Sigma A}{\Omega s}{\Omega B}\).
Taking the adjoint of the right square in~(\ref{PSdgrm})
gives a homotopy commutative diagram
\begin{equation}
  \label{adjointdgrm}
  \diagram
     \Sigma\Omega X\wedge\Omega X\rto^-{E}
        & \Omega(\Sigma^{2}\Omega X\wedge\Omega X)
             \rrto^-{\Omega\Sigma\mu^{\ast}}
             \dto^{\Omega(ev\wedge ev)\circ\Omega\Sigma\tau}
        & & \Omega\Sigma^{2}\Omega X\dto^{\Omega c} \\
     & \Omega(X\wedge X)\rrto & & \Omega C(X).
  \enddiagram
\end{equation}
Consider both directions around the diagram. The naturality of $E$
implies that the upper direction around~(\ref{adjointdgrm}) is homotopic
to $\Omega c\circ E\circ\mu^{\ast}$. But $\Omega c\circ E$
is homotopic to the adjoint of $c$, and we know from~(\ref{PSdgrm})
that the adjoint of $c$ is also homotopic to $b\circ a$. Therefore
the upper direction around~(\ref{adjointdgrm}) is homotopic to the 
composite 
\begin{equation} 
  \label{adjointeqn1} 
  \namedddright{\Sigma\Omega X\wedge\Omega X}{\mu^{\ast}}  
        {\Sigma\Omega X}{a}{F(X)}{b}{\Omega C(X)}. 
\end{equation}
On the other hand, in Lemma~\ref{PSevfactor} we show that in the 
special case $X=\Sigma Y$ the lower direction around~(\ref{adjointdgrm}) 
is homotopic to a composite 
\begin{equation} 
  \label{adjointeqn2} 
  \namedddright{\Sigma\Omega\Sigma Y\wedge\Omega\Sigma Y}{\epsilon} 
      {Y\wedge\Sigma Y}{E}{\Omega(\Sigma Y\wedge\Sigma Y)}{} 
      {\Omega C(\Sigma Y)} 
\end{equation} 
for some map $\epsilon$. 

\begin{lemma}
   \label{PSevfactor}
   Suppose $X$ is a suspension, $X=\Sigma Y$. Then
   there is a homotopy commutative diagram
   \[\diagram 
        \Sigma\Omega\Sigma Y\wedge\Omega\Sigma Y 
             \rto^-{\tau}\dto^{\epsilon} 
           & \Omega\Sigma Y\wedge\Sigma\Omega\Sigma Y\rto^-{E} 
           & \Omega(\Sigma\Omega\Sigma Y\wedge\Sigma\Omega\Sigma Y)
             \dto^{\Omega(ev\wedge ev)} \\
         Y\wedge\Sigma Y\rrto^-{E}
           & & \Omega(\Sigma Y\wedge\Sigma Y).
     \enddiagram\]
\end{lemma}

\begin{proof} 
First observe that the Hopf construction can be used to obtain 
a homotopy equivalence
\[\lllnameddright{\Sigma Y\vee(\Sigma Y\wedge\Omega\Sigma Y)}
     {1\vee(\Sigma E\wedge 1)}
     {\Sigma Y\vee(\Sigma\Omega\Sigma Y\wedge\Omega\Sigma Y)}
     {\Sigma E+\mu^{\ast}}{\Sigma\Omega\Sigma Y}.\] 
This can be applied to $\Omega\Sigma Y\wedge\Sigma\Omega\Sigma Y$  
by decomposing the right smash factor, then switching the suspension 
coordinate to the left smash factor, and decomposing again. From this 
we obtain a homotopy equivalence   
\[e:\namedright{(Y\vee\Sigma Y)\vee 
    (Y\wedge\Omega\Sigma Y\wedge\Sigma Y)\vee 
    (Y\wedge\Sigma Y\wedge\Omega\Sigma Y)\vee 
    (Y\wedge\Omega\Sigma Y\wedge\Sigma Y\wedge\Omega\Sigma Y)} 
    {}{\Omega\Sigma Y\wedge\Sigma\Omega\Sigma Y}\] 
where the first wedge summand maps in by $E\wedge\Sigma E$, 
the second maps in through $\mu^{\ast}\wedge 1$, the third 
maps in through $1\wedge\mu^{\ast}$, and the fourth maps in 
by either $\mu^{\ast}\wedge 1$ or $1\wedge\mu^{\ast}$ (the ``$1$" 
here refers to the identity map on $\Omega\Sigma Y$). 
Thus to prove the Lemma it is equivalent to show that 
$\Omega(ev\wedge ev)\circ E\circ e$ is null homotopic when 
restricted to the second, third, and fourth wedge summands. 

We will use the homotopy fibration 
\(\nameddright{\Sigma\Omega\Sigma Y\wedge\Omega\Sigma Y}{\mu^{\ast}}
     {\Sigma\Omega\Sigma Y}{ev}{Y}\), 
due to Ganea~\cite{Ga}. The naturality of $E$ implies that the composite 
\[\lnamedright{\Omega\Sigma Y\vee 
    (\Sigma\Omega\Sigma Y\wedge\Omega\Sigma Y)}{1\wedge\mu^{\ast}} 
    {\Omega\Sigma Y\wedge\Sigma\Omega\Sigma Y} 
    \stackrel{E}{\longrightarrow} 
    \lllnamedright{\Omega(\Sigma\Omega\Sigma Y\wedge\Sigma\Omega\Sigma Y)} 
    {\Omega(ev\wedge ev)}{\Omega(\Sigma Y\wedge\Sigma Y)}\] 
is homotopic to 
$\Omega(ev\wedge ev)\circ\Omega\Sigma(1\wedge\mu^{\ast})\circ E$. 
But the latter is null homotopic as $ev\circ\mu^{\ast}$ is 
null homotopic. Thus $\Omega(ev\wedge ev)\circ E\circ e$ is null 
homotopic when restricted to the second and fourth wedge summands. 
Similarly, if the suspension coordinate is switched to the left 
smash factor, then $\Omega(ev\wedge ev)\circ E\circ(\mu^{\ast} 1)$ 
is null homotopic and so $\Omega(ev\wedge ev)\circ E\circ e$ is 
null homotopic when restricted to the third (and fourth) wedge summand. 
\end{proof}

Using the homotopies in~(\ref{adjointeqn1}) and~(\ref{adjointeqn2}) 
to replace the upper and lower directions in~(\ref{adjointdgrm}) 
when $X=\Sigma Y$, we obtain a homotopy commutative diagram
\begin{equation}
  \label{adjointdgrm2}
  \diagram
       \Sigma\Omega\Sigma Y\wedge\Omega\Sigma Y
             \rto^-{\mu^{\ast}}\dto^{\epsilon}
          & \Sigma\Omega\Sigma Y\rto^-{a} & F(\Sigma Y)\dto^{b} \\
       Y\wedge\Sigma Y\rto^-{E} & \Omega\Sigma Y\wedge\Sigma Y\rto
          & \Omega C(\Sigma Y).
    \enddiagram
\end{equation} 
Since $\Sigma Y$ is a suspension, the reduced diagonal
\(\namedright{\Sigma Y}{\overline{\Delta}}{\Sigma Y\wedge\Sigma Y}\)
is null homotopic. Therefore there are homotopy equivalences
$F(\Sigma Y)\simeq\Sigma Y\times\Omega(\Sigma Y\wedge\Sigma Y)$
and $C(\Sigma Y)\simeq\Sigma^{2} Y\vee(\Sigma Y\wedge\Sigma Y)$.
Further, the projection of $\Omega(\Sigma Y\wedge\Sigma Y)$ off
$F(\Sigma Y)$ can be chosen to be the composite
\(\nameddright{F(\Sigma Y)}{b}{\Omega C(\Sigma Y)}{\mbox{proj}}
     {\Omega(\Sigma Y\wedge\Sigma Y)}\).
Thus~(\ref{adjointdgrm2}) implies that there is a homotopy commutative 
diagram
\begin{equation}
  \label{adjointdgrm3}
  \diagram
       \Sigma\Omega\Sigma Y\wedge\Omega\Sigma Y
             \rto^-{\mu^{\ast}}\dto^{\epsilon}
          & \Sigma\Omega\Sigma Y\rto^-{a} & F(\Sigma Y)\dto^{\pi_{2}} \\
       Y\wedge\Sigma Y\rrto^-{E} & & \Omega(\Sigma Y\wedge\Sigma Y). 
    \enddiagram
\end{equation}

Now we return to the case of interest. Take $Y=S^{n-1}$
in~(\ref{adjointdgrm3}) and combine it with the factorization in
Lemma~\ref{omegagfact} to obtain a homotopy commutative diagram
\begin{equation}
  \label{Snevdgrm}
  \diagram
     \Omega(\Sigma\Omega S^{n}\wedge\Omega S^{n})
          \rto^-{\Omega\mu^{\ast}}\dto^{\Omega\epsilon}
        & \Omega\Sigma\Omega S^{n}\rto^-{\Omega a}
        & \Omega F(S^{n})\rto^-{\Omega g}\dto^{\Omega\pi_{2}}
        & \Omega^{3} S^{2n+1}\ddouble \\
     \Omega(S^{n-1}\wedge\Sigma S^{n-1})\rrto^-{\Omega E}
        & & \Omega^{2}(S^{n}\wedge S^{n})\rto^-{\Omega g_{2}}
        & \Omega^{3} S^{2n+1}.
  \enddiagram
\end{equation} 
Two adjustments can be made to~(\ref{Snevdgrm}).
First, since $g_{2}$ is degree one in $H_{2n-2}(\ )$, the
composite $g_{2}\circ E$ is homotopic to $E^{2}$.
Second, since the homotopy fibration
\(\nameddright{\Sigma\Omega S^{n}\wedge\Omega S^{n}}{\mu^{\ast}}
    {\Sigma\Omega S^{n}}{ev}{S^{n}}\)
has a section
\(i:\namedright{S^{n}}{}{\Sigma\Omega S^{n}}\)
given by including the bottom cell, there is a homotopy equivalence
\[\llnameddright{\Omega S^{n}\times
     \Omega(\Sigma\Omega S^{n}\wedge\Omega S^{n})}
     {\Omega i\times\Omega\mu^{\ast}}
     {\Omega\Sigma\Omega S^{n}\times\Omega\Sigma\Omega S^{n}}
     {\mu}{\Omega\Sigma\Omega S^{n}}\]
where $\mu$ is the loop multiplication. By connectivity, the
composite $a\circ g\circ i$ is null homotopic, so after looping,
$\Omega g\circ\Omega a$ is homotopic to the composite
\(\Omega\Sigma\Omega S^{n}\stackrel{\pi_{2}}{\longrightarrow}
     \lllnamedright{\Omega(\Sigma\Omega S^{n}\wedge\Omega S^{n})}
     {\Omega(g\circ a\circ\mu^{\ast})}{\Omega^{3} S^{2n+1}}\), 
where $\pi_{2}$ is the projection. Combining~(\ref{Snevdgrm})
and the two adjustments we have homotopies
\[\Omega g\circ\Omega a\simeq
      \Omega g\circ\Omega a\circ\Omega\mu^{\ast}\circ\pi_{2}\simeq
      \Omega g_{2}\circ\Omega E\circ\Omega\epsilon\circ\pi_{2}\simeq
      \Omega E^{2}\circ\Omega\epsilon\circ\pi_{2}.\]
That is, there is a homotopy commutative diagram
\begin{equation}
  \label{Snevdgrm2}
  \diagram
     \Omega\Sigma\Omega S^{n}\rto^-{\Omega a}\dto^{\pi_{2}}
        & \Omega F(S^{n})\rto^-{\Omega g}
        & \Omega^{3} S^{2n+1}\ddouble \\
     \Omega(\Sigma\Omega S^{n}\wedge\Omega S^{n})
           \rto^-{\Omega\epsilon}
        & \Omega S^{2n-1}\rto^-{\Omega E^{2}} & \Omega^{3} S^{2n+1}.
  \enddiagram
\end{equation}

Finally, precompose~(\ref{Snevdgrm}) with
\(\namedright{\Omega S^{2n}}{E}{\Omega\Sigma\Omega S^{n}}\).
Let $h$ be the composite
\[h:\namedddright{\Omega S^{n}}{E}{\Omega\Sigma\Omega S^{n}}
    {\pi_{2}}{\Omega(\Sigma\Omega S^{n}\wedge\Omega S^{n})}
    {\Omega\epsilon}{\Omega S^{2n-1}}.\]
A homology calculation shows that $h_{\ast}$ is an epimorphism.
Thus a Serre spectral sequence calculation shows that the
fiber of $h$ has homology isomorphic to $\hlgy{S^{n-1}}$ and
so is homotopy equivalent to~$S^{n-1}$. Incorporating $h$ 
into~(\ref{Snevdgrm2}) and reorienting, we obtain a homotopy 
commutative diagram
\begin{equation}
  \label{Snevdgrm3}
  \diagram
     \Omega S^{n}\xto[rrr]^-{h}\ddouble
     & & & \Omega S^{2n-1}\dto^{\Omega E^{2}} \\
     \Omega S^{n}\rto^{E}
         & \Omega\Sigma\Omega S^{n}\rto^-{\Omega a}
         & \Omega F(S^{n})\rto^-{\Omega g} & \Omega^{3} S^{2n+1}.
  \enddiagram
\end{equation}

\textit{Proof of Theorem~\ref{improvedR}}:
Put~(\ref{Snevdgrm3}) together with the lower rectangle
in~(\ref{fulldgrm}) to obtain a homotopy commutative diagram
\begin{equation}
  \label{assemble}
  \diagram
     \Omega S^{n}\xto[rrr]^-{h}\ddouble
     & & & \Omega S^{2n-1}\dto^{\Omega E^{2}} \\
     \Omega S^{n}\rto^{E}\dto^{\Omega E}
         & \Omega\Sigma\Omega S^{n}
             \rto^-{\Omega a}\dto^{\Omega\Sigma\Omega E}
         & \Omega F(S^{n})\rto^-{\Omega g}\dto^{\Omega F(E)}
         & \Omega^{3} S^{2n+1}\dto \\
     \Omega^{2} S^{n+1}\rto^-{E}
         & \Omega\Sigma\Omega^{2} S^{n+1}\rto^-{\Omega a}
         & \Omega F(\Omega S^{n+1})\rto^-{\Omega f}
         & \Omega G^{2n+1}.
  \enddiagram
\end{equation}
When $n=2m$, let $S$ be the composite $\Omega f\circ\Omega a\circ E$
along the bottom row, and when $n=2m-1$ let $S^{\prime}$ be the
composite $\Omega f\circ\Omega a\circ E$ along the bottom row.
Then the outer perimeter of~(\ref{assemble}) gives both diagrams
asserted by the theorem.
$\qqed$

\section{Factoring the $4^{th}$-power map on $\Omega^{2} S^{2n+1}$
         through the double suspension}
\label{sec:4factor}

In this section we prove Theorem~\ref{4factor}.
We begin by recalling $\hlgy{\Omega S^{2n+1}\{2\}}$.
Consider the homotopy fibration sequence
\[\namedddright{\Omega^{2} S^{2n+1}}{}{\Omega S^{2n+1}\{2\}}
     {}{\Omega S^{2n+1}}{2}{\Omega S^{2n+1}}.\]
Since $2_{\ast}$ is zero in homology,
the Serre spectral sequence converging to $\hlgy{\Omega S^{2n+1}\{2\}}$
collapses at~$E_{2}$, and so there is a coalgebra isomorphism
$\hlgy{\Omega S^{2n+1}\{2\}}\cong\hlgy{\Omega S^{2n+1}}\otimes
     \hlgy{\Omega^{2} S^{2n+1}}$.
Recall that $\hlgy{\Omega S^{2n+1}}\cong\zmodtwo [a_{2n}]$ and
$\hlgy{\Omega^{2} S^{2n+1}}\cong\left(
    \otimes_{i=1}^{\infty}\Lambda (b_{2^{i}n-1})\right)
    \otimes\left(\otimes_{i=2}^{\infty}\zmodtwo [c_{2^{i}n-2}]\right)$,
where $\beta b_{2^{i}n-1}=c_{2^{i}n-2}$ for $i\geq 2$. Further,
the inclusion of the bottom Moore space
\(\namedright{P^{2n}(2)}{}{\Omega S^{2n+1}\{2\}}\)
implies that $\beta a_{2n}=b_{2n-1}$ in $\hlgy{\Omega S^{2n+1}\{2\}}$.

\begin{lemma}
   \label{Alemma}
   There is a space $A$ and a map
   \(\overline{E}\colon\namedright{A}{}{\Omega S^{2n+1}\{2\}}\)
   such that $\hlgy{A}\cong\Lambda(x_{2n-1},y_{2n})$ with
   $\beta y_{2n}=x_{2n-1}$, and $\overline{E}_{\ast}$ is an inclusion.
\end{lemma}

\begin{proof}
Let $Y=(\Omega S^{2n+1}\{2\})_{4n-1}$. A vector space basis
for $\hlgy{Y}$ is given by
\[\{b_{2n-1},a_{2n},c_{4n-2},b_{4n-1},b_{2n-1}\otimes a_{2n}\}\]
and the action of the Bockstein is given by
$\beta a_{2n}=b_{2n-1}$ and $\beta b_{4n-1}=c_{4n-2}$.
Let $Z$ be the quotient space obtained by pinching out the bottom
Moore space $P^{2n}(2)$ in $Y$. A vector space basis for $\hlgy{Z}$
is given by $\{d_{4n-2},e_{4n-1},f_{4n-1}\}$ where
these three elements are the images of
$c_{4n-2},b_{4n-1},b_{2n-1}\otimes a_{2n}$ respectively. In particular,
$\beta e_{4n-1}=d_{4n-2}$. For low dimensional reasons, the
homotopy type of $Z$ is determined by its homology, so
$Z\simeq P^{4n-1}(2)\vee S^{4n-1}$. Pinching to the Moore space
gives a composite
\(\nameddright{Y}{}{Z}{}{P^{4n-1}(2)}\)
which is onto in homology. Define the space~$F$ by the homotopy 
fibration
\[\nameddright{F}{t}{Y}{}{P^{4n-1}(2)}.\]
Let $A$ be the $(4n-1)$-skeleton of $F$. A Serre spectral sequence
calculation shows that $\hlgy{A}$ has vector space basis
$\{x_{2n-1},y_{2n},z_{4n-1}\}$ where $t_{\ast}$ sends these three
elements to $b_{2n-1},a_{2n}$, and $b_{2n-1}\otimes a_{2n}$
respectively. Since $t_{\ast}$ is a coalgebra map and commutes
with Bocksteins, we obtain a coalgebra isomorphism
$\hlgy{A}\cong\Lambda (x_{2n-1},y_{2n})$ with $\beta y_{2n}=x_{2n-1}$.
Moreover, by construction, the composite
\[\overline{E}:A=(F)_{4n-1}\hookrightarrow\namedright{F}{}
    {Y=(\Omega S^{2n+1}\{2\})_{4n-1}}\hookrightarrow\Omega S^{2n+1}\{2\}.\]
is an inclusion in homology.
\end{proof}

\begin{lemma}
   \label{tgtinclusion}
   There is a homotopy fibration
   \(\nameddright{S^{2n-1}}{}{A}{}{S^{2n}}\)
   and a homotopy fibration diagram
   \[\diagram
        S^{2n-1}\rto\dto^{E^{2}} & A\rto\dto^-{\overline{E}}
            & S^{2n}\dto^{E} \\
        \Omega^{2} S^{2n+1}\rto
            & \Omega S^{2n+1}\{2\}\rto & \Omega S^{2n+1}.
     \enddiagram\]
\end{lemma}

\begin{proof}
Since $A$ has dimension $4n-1$, the composite
\(\nameddright{A}{\overline{E}}{\Omega S^{2n+1}\{2\}}{}{\Omega S^{2n+1}}\)
factors through the $(4n-1)$-skeleton of $\Omega S^{2n+1}$, which
is $S^{2n}$, giving a map
\(f:\namedright{A}{}{S^{2n}}\)
which makes the right square in the statement of the lemma
homotopy commute. Since $\overline{E}_{\ast}$ is an inclusion,
$f_{\ast}$ is onto. A Serre spectral sequence calculation shows
that the homotopy fiber of $f$ has homology isomorphic to $\hlgy{S^{2n-1}}$
and so the fiber is homotopy equivalent to $S^{2n-1}$. Again,
since $\overline{E}_{\ast}$ is onto, the induced map of fibers
\(\namedright{S^{2n-1}}{}{\Omega^{2} S^{2n-1}}\)
is degree one in $H_{2n-1}(\ )$ and so is homotopic to $E^{2}$.
\end{proof}

Let
\(\partial_{A}:\namedright{\Omega S^{2n}}{}{S^{2n-1}}\)
be the connecting map for the homotopy fibration
\(\nameddright{S^{2n-1}}{}{A}{}{S^{2n}}\).
Continuing the homotopy fibration diagram in Lemma~\ref{tgtinclusion}
one step to the left proves the following.

\begin{corollary}
   \label{2partial}
   There is a homotopy commutative diagram
   \[\diagram
        \Omega S^{2n}\rto^-{\partial_{A}}\dto^{\Omega E}
            & S^{2n-1}\dto^{E^{2}} \\
        \Omega^{2} S^{2n+1}\rto^-{2} & \Omega^{2} S^{2n+1}.
     \enddiagram\]
   $\qqed$
\end{corollary}

\noindent
\textit{Proof of Theorem~\ref{4factor}:}
By~\cite[\S 4]{C}, the composite
\(\nameddright{\Omega^{2} S^{2n+1}}{2}{\Omega^{2} S^{2n+1}}{\Omega H}
     {\Omega^{2} S^{4n+1}}\)
is null homotopic, so there is a lift of $2$ through $\Omega E$
to a map
\(\namedright{\Omega^{2} S^{2n+1}}{}{\Omega S^{2n}}\).
Combining this with Corollary~\ref{2partial} gives a homotopy
commutative diagram
\[\diagram
     & \Omega S^{2n}\rto^-{\partial_{A}}\dto^{\Omega E}
          & S^{2n-1}\dto^{E^{2}} \\
     \Omega^{2} S^{2n+1}\rto^-{2}\urto & \Omega^{2} S^{2n+1}\rto^-{2}
          & \Omega^{2} S^{2n+1}
  \enddiagram\]
which proves the theorem.
$\qqed$
\medskip

\begin{remark}
The space $A$ is similar to the unit tangent bundle $\tau(S^{2n})$.
They both have the same homology over the Steenrod algebra, and both
are total spaces in fibrations over $S^{2n}$ with fiber $S^{2n-1}$.
It would be curious to know whether $A$ is homotopy equivalent
to $\tau(S^{2n})$.
\end{remark}

\section{The $4^{th}$-power map on $\Omega W_{n}$}
\label{sec:expWn} 

In this section we prove Theorem~\ref{Wnpower}. To do so we will
use Theorem~\ref{4factor} and Corollary~\ref{Richter} along with
several other known results collected in Theorems~\ref{loop2facts} 
and~\ref{BWn}.  

For $r\geq 1$, let 
\(\underline{2}^{r}\colon\namedright{S^{2n+1}}{}{S^{2n+1}}\) 
be the map of degree~$2^{r}$, and let \curly\ be 
its homotopy fiber. Let  
\(2^{r}\colon\namedright{\Omega S^{2n+1}}{2^{r}}{\Omega S^{2n+1}}\) 
be the $2^{r}$-power map, and let $\Omega S^{2n+1}\{2^{r}\}$ be 
its homotopy fiber. Note that the $2^{r}$-power map need not be an 
$H$-space so $\Omega S^{2n+1}\{2^{r}\}$ need not be an $H$-space. 
Theorem~\ref{loop2facts} summarizes some of the information 
from~\cite[\S 4]{C} regarding the difference 
$\Omega\underline{2}^{r}-2^{r}$.    

\begin{theorem}
   \label{loop2facts}
   The following hold:
   \begin{letterlist}
      \item the composite 
            \(\llnameddright{\Omega S^{2n+1}}{\Omega\underline{2}-2}
                {\Omega S^{2n+1}}{\Omega H}{\Omega^{2} S^{4n+1}}\)
            is null homotopic; 
      \item after looping, $\Omega^{2}\underline{2}^{r}-2^{r}$ has order $2$,
            and there are homotopies
            $2^{r+1}\simeq 2\Omega^{2}\underline{2}^{r}\simeq 
               \Omega^{2}\underline{2}^{r+1}$. 
   \end{letterlist}
   $\qqed$
\end{theorem}

Next, from the double suspension we obtain a homotopy fibration
\(\nameddright{W_{n}}{}{S^{2n-1}}{E^{2}}{\Omega^{2} S^{2n+1}}\)
which defines the space $W_{n}$. Gray~\cite{Gr} showed that $W_{n}$
has a classifying space $BW_{n}$ which fits nicely with the
$EHP$-sequence. Specifically, he proved the following.

\begin{theorem}
   \label{BWn}
   There are homotopy fibrations
   \[\nameddright{S^{2n-1}}{E^{2}}{\Omega^{2} S^{2n+1}}{\nu}{BW_{n}}\]
   \[\nameddright{BW_{n}}{j}{\Omega^{2} S^{4n+1}}{\phi}{S^{4n-1}}\]
   satisfying the following properties:
   \begin{letterlist}
      \item the composite
            \(\nameddright{\Omega^{2} S^{2n+1}}{\nu}{BW_{n}}{j}
                 {\Omega^{2} S^{4n+1}}\)
            is homotopic to the looped Hopf invariant $\Omega H$;
      \item the composite
            \(\nameddright{S^{4n-1}}{E^{2}}{\Omega^{2} S^{4n+1}}
                 {\phi}{S^{4n-1}}\)
            is homotopic to the degree $2$ map;
      \item the composite
            \(\nameddright{BW_{n}}{j}{\Omega^{2} S^{4n+1}}{2}
                 {\Omega^{2} S^{4n+1}}\)
            is null homotopic;
      \item after looping,
            \(\namedright{\Omega^{3} S^{4n+1}}{\Omega\phi}
                 {\Omega S^{4n-1}}\)
            is homotopic to the composite
            \(\nameddright{\Omega^{3} S^{4n+1}}{\Omega P}
                 {\Omega S^{2n}}{H}{\Omega S^{4n-1}}\).
   \end{letterlist}
   $\qqed$
\end{theorem}

\noindent
\textit{Proof of Theorem~\ref{Wnpower}:}
The proof is divided into two cases, one for $\Omega W_{2n}$ and
the other for $\Omega W_{2n-1}$.

\noindent
\textbf{Case 1:} $\mathbf{\Omega W_{2n}}$. Consider the homotopy
fibration sequence
\[\namedddright{\Omega^{3} S^{4n+1}}{\Omega\nu}{W_{2n}}{g}{S^{4n-1}}
    {E^{2}}{\Omega^{2} S^{4n+1}}\]
where the left map has been identified as a loop map by Theorem~\ref{BWn}.
By Theorem~\ref{BWn}~(b), the degree~$2$ map on $S^{4n-1}$ factors
through $E^{2}$, so the composite
\(\nameddright{W_{2n}}{g}{S^{4n-1}}{\underline{2}}{S^{4n-1}}\)
is null homotopic. On the other hand, after looping we have
$2\circ\Omega g\simeq\Omega g\circ 2$. Thus
$(\Omega\underline{2}-2)\circ\Omega g\simeq -2\Omega g$.
By Theorem~\ref{loop2facts}~(a),
$H\circ (\Omega\underline{2}-2)\simeq\ast$.
Therefore $H\circ (-2\Omega g)\simeq\ast$, and so $-2\Omega g$
lifts to the fiber of~$H$ to give a homotopy commutative diagram
\begin{equation}
  \label{W2ndgrm1}
  \diagram
       \Omega W_{2n}\rto^-{-2}\dto^{\lambda}
            & \Omega W_{2n}\dto^{\Omega g} \\
       S^{4n-2}\rto^-{E} & \Omega S^{4n-1}
  \enddiagram
\end{equation}
for some map $\lambda$.

Next, consider the diagram
\[\diagram
      & \Omega^{3} S^{4n+1}\rto^-{\Omega\nu}\dto^{2}\dlto & W_{2n}\dto^{2} \\
      \Omega S^{4n-1}\rto^-{\Omega E^{2}}
          & \Omega^{3} S^{4n+1}\rto^-{\Omega\nu} & W_{2n}.
  \enddiagram\]
The left triangle homotopy commutes by Corollary~\ref{Richter}.
The right square homotopy commutes since~$\Omega\nu$ is a loop map.
The bottom row is two consecutive maps in a homotopy fibration
and so is null homotopic. Thus $2\circ\Omega\nu\simeq\ast$, implying
that there is a homotopy fibration diagram
\begin{equation}
  \label{W2ndgrm2}
  \diagram
     \Omega W_{2n}\rto^-{\Omega g}\dto^{2}
         & \Omega S^{4n-1}\rto^-{\Omega E^{2}}\dto^{\gamma}
         & \Omega^{3} S^{4n+1}\rto^-{\Omega\nu}\dto & W_{2n}\dto^{2} \\
     \Omega W_{2n}\rdouble & \Omega W_{2n}\rto & \ast\rto & W_{2n}
  \enddiagram
\end{equation}
for some map $\gamma$.

Juxtaposing~(\ref{W2ndgrm1}) and the left most square in~(\ref{W2ndgrm2})
gives a homotopy commutative diagram
\[\diagram
     \Omega W_{2n}\rto^-{-2}\dto^{\lambda}
         & \Omega W_{2n}\rto^-{2}\dto^{\Omega g}
         & \Omega W_{2n}\ddouble \\
     S^{4n-2}\rto^-{E} & \Omega S^{4n-1}\rto^-{\gamma} & \Omega W_{2n}.
  \enddiagram\]
The lower row is null homotopic since $\Omega W_{2n}$ is $(8n-5)$-connected.
Thus $-4$ is null homotopic on $\Omega W_{2n}$, and so $4$ is null homotopic
as well.
\medskip

\noindent
\textbf{Case 2:} $\mathbf{\Omega W_{2n-1}}$.
In this case we do not have a factorization of the degree $2$ map
on $S^{4n-3}$ through the double suspension as we did for $S^{4n-1}$
in Case~1, so we have to work harder. We begin with some general
statements which are valid for any $W_{n}$. Since the map
\(\namedright{\Omega^{2} S^{2n+1}}{\Omega^{2}\underline{2}}
    {\Omega^{2} S^{2n+1}}\)
is degree two on the bottom cell, there is a homotopy fibration diagram
\begin{equation}
  \label{psidef}
  \diagram
     W_{n}\rto^-{g}\dto^{\psi}
         & S^{2n-1}\rto^-{E^{2}}\dto^{\underline{2}}
         & \Omega^{2} S^{2n+1}\dto^-{\Omega^{2}\underline{2}} \\
     W_{n}\rto^-{g} & S^{2n-1}\rto^-{E^{2}} & \Omega^{2} S^{2n+1}
  \enddiagram \
\end{equation}
for some map $\psi$. We will show that $\psi\circ\psi$ is null homotopic
using an argument due to~\cite{CMN1}. By Theorem~\ref{loop2facts}~(b),
$\Omega^{2}\underline{4}\simeq 4$ as self-maps of $\Omega^{2} S^{2n+1}$,
and by Theorem~\ref{4factor}, the $4^{th}$-power map on $\Omega^{2} S^{2n+1}$
factors through the double suspension. Thus the two fibration diagrams
\[\diagram
      W_{n}\rto^-{g}\dto^{\psi}
         & S^{2n-1}\rto^-{E^{2}}\dto^{\underline{2}}
         & \Omega^{2} S^{2n+1}\dto^{\Omega^{2}\underline{2}}
         & & W_{n}\rto^-{g}\dto
         & S^{2n-1}\rto^-{E^{2}}\dto^{\underline{4}}
         & \Omega^{2} S^{2n+1}\dto \\
      W_{n}\rto^-{g}\dto^{\psi}
         & S^{2n-1}\rto^-{E^{2}}\dto^{\underline{2}}
         & \Omega^{2} S^{2n+1}\dto^{\Omega^{2}\underline{2}}
         & & \ast\rto\dto & S^{2n-1}\rto^-{=}\dto^{=}
         & S^{2n-1}\dto^{E^{2}} \\
      W_{n}\rto^-{g} & S^{2n-1}\rto^-{E^{2}}
         & \Omega^{2} S^{2n+1} & & W_{n}\rto^-{g}
         & S^{2n-1}\rto^-{E^{2}} & \Omega^{2} S^{2n+1}
  \enddiagram\]
give two choices of a morphism of homotopy fibrations which covers
the $4^{th}$-power map on $\Omega^{2} S^{2n+1}$. The obstructions
to the two morphisms being fiber homotopic lie in
$H^{2n-1}(S^{2n-1};\pi_{2n-1}(W_{n}))$. But $\pi_{2n-1}(W_{n})=0$ by
connectivity, so the obstructions vanish. Hence $\psi\circ\psi\simeq\ast$.

Now specialize to $W_{2n-1}$.
We proceed for a while as in Case~$1$. From~(\ref{psidef}) we have
$\Omega\underline{2}\circ\Omega g\simeq\Omega g\circ\Omega\psi$.
On the other hand, we have
$2\circ\Omega g\simeq\Omega g\circ 2$.
Substracting then gives a homotopy commutative diagram
\[\diagram
      \Omega W_{2n-1}\rto^-{\Omega g}\dto^{\Omega\psi -2}
         & \Omega S^{4n-3}\dto^{\Omega\underline{2}-2} \\
      \Omega W_{2n-1}\rto^-{\Omega g} & \Omega S^{4n-3}.
  \enddiagram\]
By Theorem~\ref{loop2facts}~(a), $H\circ (\Omega\underline{2}-2)\simeq\ast$.
Therefore
$H\circ\Omega g\circ(\Omega\psi-2)\simeq
    H\circ(\Omega\underline{2}-2)\circ\Omega g\simeq\ast$,
and so $\Omega g\circ(\Omega\psi-2)$ lifts to the fiber of $H$ to
give a homotopy commutative diagram
\begin{equation}
  \label{W2n-1dgrm1}
  \diagram
       \Omega W_{2n-1}\rto^-{\Omega\psi -2}\dto^{\lambda}
            & \Omega W_{2n-1}\dto^{\Omega g} \\
       S^{4n-4}\rto^-{E} & \Omega S^{4n-3}
  \enddiagram
\end{equation}
for some map $\lambda$.

Next, consider the diagram
\[\diagram
      & \Omega^{3} S^{4n-1}\rto^-{\Omega\nu}
            \dto^{\Omega^{3}\underline{2}-2}\dlto
          & W_{2n-1}\dto^{\psi -2} \\
      \Omega S^{4n-3}\rto^-{\Omega E^{2}}
          & \Omega^{3} S^{4n-1}\rto^-{\Omega\nu} & W_{2n-1}.
  \enddiagram\]
The left triangle homotopy commutes by Corollary~\ref{Richter}.
The right square homotopy commutes since
$\Omega\nu\circ 2\simeq 2\circ\Omega\nu$ and, from~(\ref{psidef}),
$\Omega\nu\circ\Omega^{3}\underline{2}\simeq\psi\circ\Omega\nu$.
The bottom row is two consecutive maps in a homotopy fibration
and so is null homotopic. Thus $(\psi -2)\circ\Omega j\simeq\ast$,
implying that there is a homotopy fibration diagram
\begin{equation}
  \label{W2n-1dgrm2}
   \diagram
     \Omega W_{2n-1}\rto^-{\Omega g}\dto^{\Omega\psi -2}
         & \Omega S^{4n-3}\rto^-{\Omega E^{2}}\dto^{\gamma}
         & \Omega^{3} S^{4n-1}\rto^-{\Omega\nu}\dto
         & W_{2n-1}\dto^{\psi -2} \\
     \Omega W_{2n-1}\rdouble & \Omega W_{2n-1}\rto & \ast\rto & W_{2n-1}
  \enddiagram
\end{equation}
for some map $\gamma$.

Juxtaposing the leftmost squares in~(\ref{W2n-1dgrm1})
and~(\ref{W2n-1dgrm2}) gives a homotopy commutative diagram
\[\diagram
     \Omega W_{2n-1}\rto^-{\Omega\psi -2}\dto^{\lambda}
         & \Omega W_{2n-1}\rto^-{\Omega\psi -2}\dto^{\Omega g}
         & \Omega W_{2n-1}\ddouble \\
     S^{4n-4}\rto^-{E} & \Omega S^{4n-3}\rto^-{\gamma}
         & \Omega W_{2n-1}.
  \enddiagram\]
The lower row is null homotopic since $\Omega W_{2n-1}$ is
$(8n-9)$-connected. Thus
$(\Omega\psi -2)\circ (\Omega\psi -2)$ is null homotopic.
Distributing, $\Omega\psi\circ\Omega\psi - 4\Omega\psi +4$ is
null homotopic. We have already seen that $\psi\circ\psi$ is
null homotopic, so if we can show that $4\Omega\psi$ is null
homotopic then the $4^{th}$-power map on $\Omega W_{2n-1}$ is
null homotopic and we are done.

It remains to show that $4\Omega\psi$ is null homotopic. For this
we use another obstruction theory argument. By Theorem~\ref{BWn}~(d),
there is a homotopy pullback
\[\diagram
     W_{n}\rto^-{g}\dto^{\Omega\nu}
        & S^{2n-1}\rto^-{E^{2}}\dto^{E}
        & \Omega^{2} S^{2n+1}\ddouble \\
     \Omega^{3} S^{4n+1}\rto^-{\Omega P}\dto^{\Omega\phi}
        & \Omega S^{2n}\rto^-{\Omega E}\dto^{H}
        & \Omega^{2} S^{2n+1} \\
      \Omega S^{4n-1}\rdouble & \Omega S^{4n-1}. &
  \enddiagram\]
The composite
\(\nameddright{S^{2n-1}}{\underline{2}}{S^{2n-1}}{E}{\Omega S^{2n}}\)
is homotopic to the composite
\(\nameddright{S^{2n-1}}{E}{\Omega S^{2n}}{2}{\Omega S^{2n}}\).
Thus the fibration diagrams
\[\diagram
      W_{n}\rto^-{g}\dto^{\psi}
         & S^{2n-1}\rto^-{E^{2}}\dto^{\underline{2}}
         & \Omega^{2} S^{2n+1}\dto^{2}
         & & W_{n}\rto^-{g}\dto^{\Omega\nu}
         & S^{2n-1}\rto^-{E^{2}}\dto^{E}
         & \Omega^{2} S^{2n+1}\ddouble \\
      W_{n}\rto^-{g}\dto^{\Omega\nu} & S^{2n-1}\rto^-{E^{2}}\dto^{E}
         & \Omega^{2} S^{2n+1}\ddouble
         & & \Omega^{3} S^{4n+1}\rto^-{\Omega P}\dto^{2}
         & \Omega S^{2n}\rto^-{\Omega E}\dto^{2}
         & \Omega^{2} S^{2n+1}\dto^{2} \\
      \Omega^{3} S^{4n+1}\rto^-{\Omega P} & \Omega S^{2n}\rto^-{\Omega E}
         & \Omega^{2} S^{2n+1}
         & & \Omega^{3} S^{4n+1}\rto^-{\Omega P}
         & \Omega S^{2n}\rto^-{\Omega E} & \Omega^{2} S^{2n+1}
  \enddiagram\]
give two choices of a morphism of homotopy fibrations which covers
the $2^{nd}$-power map on $\Omega^{2} S^{2n+1}$. The obstructions
to these morphisms being fiber homotopic lie in
$H^{2n-1}(S^{2n-1};\pi_{2n-1}(\Omega^{3} S^{4n+1}))$. Since
$\pi_{2n-1}(\Omega^{3} S^{4n+1})=0$ by connectivity, the
obstructions vanish. Hence $\Omega\nu\circ\psi\simeq 2\circ\Omega\nu$.
By Theorem~\ref{BWn}~(c), $2\circ\Omega\nu$ is null homotopic,
so there is a lift
\[\diagram
     & W_{n}\dto^{\psi}\dldashed|>\tip_{l} & \\
     \Omega^{2} S^{4n-1}\rto^-{\Omega i} & W_{n}\rto^-{\Omega\nu}
         & \Omega^{3} S^{4n+1}
  \enddiagram\]
for some map $l$. The connecting map for the fibration along the
lower row is $\Omega^{2}\phi$. By Theorem~\ref{BWn}~(b),
$\phi\circ E^{2}$ is homotopic to the degree~$2$ map on $S^{4n-1}$. Thus
$\Omega i\circ\Omega^{2}\underline{2}$ is null homotopic. Now multiply
by $2$. By Theorem~\ref{loop2facts}~(b), $4\simeq2\Omega^{2}\underline{2}$,
and so $4\Omega i\simeq 2(\Omega i\circ\Omega^{2}\underline{2})$ is null
homotopic. Thus $\Omega i$ has order~$4$ and so $4\circ\psi$ is null
homotopic. Hence $4\Omega\psi$ is null homotopic, as required.
$\qqed$

\section{The construction of the space \anick} 
\label{sec:proof} 

In this section we prove Theorem~\ref{main}, deferring some details 
to later. We begin by defining a space and determining some maps 
which will establish the context in which to work. 

Define the space $Y$ and the map $t$ by the homotopy fibration 
\[\nameddright{Y}{t}{\Omega^{2} S^{2n+1}}{S}{\Omega^{2} S^{4n+1}\{2\}}\] 
where $S$ is the map appearing in Theorem~\ref{improvedR}. 

\begin{lemma} 
   \label{Yfib} 
   There is a homotopy fibration 
   \(\nameddright{S^{2n-1}}{f}{Y}{g}{\Omega W_{2n}}\) 
   for some map $g$, with the property that the composite 
   \(\nameddright{S^{2n-1}}{f}{Y}{t}{\Omega^{2} S^{2n+1}}\) 
   is homotopic to $E^{2}$. 
\end{lemma} 

\begin{proof} 
The left diagram in Theorem~\ref{improvedR} implies that there is a 
homotopy pullback diagram  
\[\diagram 
     Y\rto\ddouble 
       & \Omega S^{2n}\rto^-{\Omega E^{2}\circ h}\dto^{\Omega E} 
       & \Omega^{3} S^{2n+1}\dto \\ 
     Y\rto^-{t} & \Omega^{2} S^{2n+1}\rto^-{S}\dto^{\Omega H} 
       & \Omega^{2} S^{4n+1}\{2\}\dto \\ 
     & \Omega^{2} S^{4n+1}\rdouble & \Omega^{2} S^{4n+1}. 
  \enddiagram\] 
The fact that $Y$ is the homotopy fiber of the composite 
$\Omega E^{2}\circ h$ implies there is a homotopy fibration diagram 
\[\diagram 
     S^{2n-1}\rto^-{f}\ddouble & Y\rto^-{g}\dto^{t} 
        & \Omega W_{2n}\dto \\ 
     S^{2n-1}\rto^-{E} & \Omega S^{2n}\rto^-{h}\dto^{\Omega E^{2}\circ h} 
        & \Omega S^{4n-2}\dto^{\Omega E^{2}} \\ 
     & \Omega^{3} S^{4n+1}\rdouble & \Omega^{3} S^{4n+1}  
  \enddiagram\] 
which defines the maps $f$ and $g$ proves the existence of the asserted 
homotopy fibration. Since $t$ is degree one in $H_{2n-1}(\ )$, so 
is $t\circ f$, and so $t\circ f\simeq E^{2}$.   
\end{proof} 

\begin{remark} 
At odd primes one can do a little better. The corresponding map 
\(\namedright{\Omega^{2} S^{2n+1}}{S}{\Omega^{2} S^{2np+1}\{p\}}\) 
is known to factor through a map 
\(\namedright{BW_{n}}{\overline{S}}{\Omega^{2} S^{2np+1}\{p\}}\) 
whose homotopy fiber is $\Omega W_{np}$, which implies there is an  
analogous homotopy fibration 
\(\nameddright{S^{2n-1}}{}{Y}{}{\Omega W_{np}}\). 
Such a factorization of $S$ is not known at the prime $2$, but 
Lemma~\ref{Yfib} says that the homotopy fibration for~$Y$ exists 
for other reasons.  
\end{remark} 

The idea behind the proof of Theorem~\ref{main} is to find 
successive lifts 
\begin{equation} 
  \label{twolifts} 
  \diagram 
      & & \Omega^{2} S^{2n+1}\ddto^{2^{r}}
            \ddlto^{\ell_{1}}\ddllto_{\ell_{2}} \\ 
      & & \\ 
      S^{2n-1}\rto^-{f} & Y\rto^-{t} & \Omega^{2} S^{2n+1}  
  \enddiagram 
\end{equation} 
where $\ell_{1}$ exists for $r\geq 2$ and $\ell_{2}$ exists 
for $r\geq 3$. The two lifts are constructed as consequences 
of certain extensions. 

To describe how $\ell_{1}$ comes about we need to introduce several 
new spaces and maps. The Moore space \moore\ is the homtopy cofiber 
of the degree $2^{r}$ map on $S^{2n}$. Consider the pinch map 
\(\namedright{\moore}{q}{S^{2n+1}}\) 
onto the top cell. This factors as the composite 
\(\nameddright{\moore}{i}{\curly}{}{S^{2n+1}}\) 
where $i$ is the inclusion of the bottom Moore space. The 
factorization determines a homotopy pullback diagram 
\begin{equation} 
  \label{EFdgrm} 
  \diagram  
       E\rto\ddouble & F\rto^-{\imath}\dto & \Omega S^{2n+1}\dto \\ 
       E\rto & \moore\rto^-{i}\dto^{q} 
         & \curly\dto \\ 
       & S^{2n+1}\rdouble & S^{2n+1} 
  \enddiagram 
\end{equation} 
which defines the spaces $E$ and $F$ and the map $\imath$. Taking 
vertical connecting maps gives a factorization 
of $\Omega\underline{2}^{r}$ on $\Omega S^{2n+1}$ through $\imath$.  
This induces an extended homotopy pullback diagram 
\begin{equation} 
  \label{EFdgrm2} 
  \diagram 
     \Omega^{2} S^{2n+1}\rto\ddouble & \Omega\curly\rto\dto 
        & \Omega S^{2n+1}\rto^-{\Omega\underline{2}^{r}}\dto 
        & \Omega S^{2n+1}\ddouble \\  
     \Omega^{2} S^{2n+1}\rto^-{\partial_{E}} & E\rto 
        & F\rto^-{\imath} & \Omega S^{2n+1} 
  \enddiagram 
\end{equation} 
which defines the map $\partial_{E}$. In~\cite{T} it was shown 
that in the analogous odd primary case there is an extension 
\[\diagram 
      \Omega^{2} S^{2n+1}\rto^-{\partial_{E}}\dto^{S} 
          & E\dldashed|>\tip^-{e_{1}} \\ 
      \Omega^{2} S^{2np+1}\{p\}  & 
  \enddiagram\] 
for some map $e_{1}$. However, such an extension cannot exist 
in the $2$-primary case due to an obstruction of order~$2$, as 
will be described in Section~\ref{sec:deferredproofs}. Instead, 
we have to modify by introducing an extra factor of $2$.   

Let $\overline{\imath}$ be the composite 
\(\overline{\imath}\colon\nameddright{F}{\imath}{\Omega S^{2n+1}} 
    {\Omega\underline{2}}{\Omega S^{2n+1}}\). 
Define the space $\overline{E}$ by the homotopy pullback diagram 
\[\diagram 
     E\rto\ddouble & \overline{E}\rto\dto 
        & \Omega S^{2n+1}\{\underline{2}\}\dto \\ 
     E\rto & F\rto^-{\imath}\dto^-{\overline{\imath}}  
        & \Omega S^{2n+1}\dto^{\Omega\underline{2}} \\ 
     & \Omega S^{2n+1}\rdouble & \Omega S^{2n+1}. 
  \enddiagram\] 
To see the effect on $\partial_{E}$, observe that the factorization 
of the map 
\(\namedright{E}{}{F}\) 
through $\overline{E}$ implies there is a homotopy pullback diagram 
\[\diagram 
     \Omega^{2} S^{2n+1}\rto^-{\Omega^{2}\underline{2}}\dto^{\partial_{E}} 
        & \Omega^{2} S^{2n+1}\rto\dto^{\partial_{\,\overline{E}}} 
        & \Omega\curly\ddouble \\  
     E\rto\dto & \overline{E}\rto\dto & \Omega\curly \\ 
     F\rdouble & F & 
  \enddiagram\] 
which defines the map $\partial_{\,\overline{E}}$. We prove the 
following. 

\begin{proposition} 
   \label{Ebarextension} 
   If $r\geq 2$, there is an extension 
   \[\diagram 
         \Omega^{2} S^{2n+1}\rto^-{\partial_{\,\overline{E}}}\dto^{S} 
             & \overline{E}\dldashed|>\tip^-{e_{1}} \\ 
         \Omega^{2} S^{4n+1}\{2\}  & 
     \enddiagram\] 
   for some map $e_{1}$. 
\end{proposition} 

The proof is a modification of~\cite[2.1]{T} and will be described 
more fully in Section~\ref{sec:deferredproofs}. In~(\ref{EFdgrm2}) 
we saw that $\partial_{E}$ factors through $\Omega\curly$. The switch 
to $\overline{E}$ has the effect of factoring $\partial_{\,\overline{E}}$ 
through $\Omega\curlyplus$. To see this, observe that the factorization of 
\(\namedright{\Omega S^{2n+1}}{\Omega\underline{2}^{r}}{\Omega S^{2n+1}}\) 
through $\imath$ in~(\ref{EFdgrm2}) together with the extra copy of 
$\Omega\underline{2}$ in the definition of $\overline{\imath}$ implies 
that there is an extended homotopy pullback diagram 
\begin{equation} 
  \label{omegadgrm} 
  \diagram 
     & \Omega\moore\rdouble\dto^{\omega} 
          & \Omega\moore\dto^{\Omega q} & \\ 
     \Omega^{2} S^{2n+1}\rto\ddouble & \Omega\curlyplus\rto\dto  
          & \Omega S^{2n+1}\rto^-{\Omega\underline{2}^{r+1}}\dto 
          & \Omega S^{2n+1}\ddouble \\ 
     \Omega^{2} S^{2n+1}\rto^-{\partial_{\,\overline{E}}} 
          & \overline{E}\rto 
          & F\rto^-{\overline{\imath}} & \Omega S^{2n+1}  
  \enddiagram 
\end{equation}  
which defines the map $\omega$. Note that the restriction of $\omega$ 
to the bottom Moore space $P^{2n}(2^{r})$ is degree one on the top 
cell and degree~$2$ on the bottom cell. 

Making use of the factorization of $\partial_{\,\overline{E}}$ through 
$\Omega\curlyplus$, let $e^{\prime}_{1}$ be the composite 
\[e^{\prime}_{1}\colon\nameddright{\Omega\curlyplus}{}{\overline{E}} 
     {e_{1}}{\Omega^{2} S^{4n+1}\{2\}}.\]  
Then Proposition~\ref{Ebarextension} implies that there is a homotopy 
commutative square 
\[\diagram 
     \Omega^{2} S^{2n+1}\rto\dto^{S} 
        & \Omega\curlyplus\dto^-{e^{\prime}_{1}} \\ 
     \Omega^{2} S^{4n+1}\{2\}\rdouble & \Omega^{2} S^{4n+1}\{2\}. 
  \enddiagram\]  
This square determines an extended homotopy pullback diagram 
\begin{equation} 
  \label{l1dgrm} 
  \diagram 
     \Omega^{2} S^{2n+1}\rto^-{\ell_{1}}\ddouble & Y\rto\dto^{t}  
        & X\rto^-{q_{X}}\dto & \Omega S^{2n+1}\ddouble \\ 
     \Omega^{2} S^{2n+1}\rto^-{\Omega^{2}\underline{2}^{r+1}} 
        & \Omega^{2} S^{2n+1}\rto\dto^{S} 
        & \Omega\curlyplus\rto\dto^{e^{\prime}_{1}} 
        & \Omega S^{2n+1} \\ 
     & \Omega^{2} S^{2np+1}\{p\}\rdouble & \Omega^{2} S^{2np+1}\{p\} &   
  \enddiagram 
\end{equation} 
which defines the space $X$ and the maps $q_{X}$ and $\ell_{1}$. In 
particular, by Theorem~\ref{loop2facts}~(b) we have  
$\Omega^{2}\underline{2}^{r+1}\simeq 2^{r+1}$, 
and so $\ell_{1}$ is a choice of a lift of the $2^{r+1}$-power map with the 
additional property that it is a connecting map of a homotopy fibration. 

Further, the factorization of $e^{\prime}_{1}$ through $\overline{E}$ 
determines a homotopy pullback diagram 
\begin{equation} 
  \label{omegaXdgrm} 
  \diagram 
      \Omega\moore\rto^-{\omega_{X}}\ddouble & X\rto\dto & R\dto \\ 
      \Omega\moore\rto^-{\omega}  
         & \Omega\curlyplus\rto\dto^-{e^{\prime}_{1}} 
         & \overline{E}\dto^{e_{1}} \\ 
      & \Omega^{2} S^{4n+1}\{2\}\rdouble & \Omega^{2} S^{4n+1}\{2\}   
  \enddiagram 
\end{equation}  
which defines the space $R$ and the map $\omega_{X}$. Note that the 
factorizations of $\omega$ in~(\ref{omegaXdgrm}) and~(\ref{omegadgrm}) 
imply that the composite 
\(\nameddright{\Omega\moore}{\omega_{X}}{X}{q_{X}}{\Omega S^{2n+1}}\) 
is homotopic to~$\Omega q$. 

Consider the homotopy fibration 
\(\nameddright{Y}{}{X}{q_{X}}{\Omega S^{2n+1}}\) 
and the map 
\(\namedright{Y}{g}{\Omega W_{2n}}\) 
from Lemma~\ref{Yfib}. 

\begin{proposition} 
   \label{Xextension}  
   If $r\geq 2$, there is an extension 
   \[\diagram  
         Y\rto\dto^{g} & X\dldashed|>\tip^-{e_{2}} \\
         \Omega W_{2n} &
     \enddiagram\]  
   for some map $e_{2}$. 
\end{proposition} 

Proposition~\ref{Xextension} implies that for $r\geq 2$ there 
is an extended homotopy pullback diagram    
\begin{equation} 
  \label{Tpb} 
  \diagram 
     \Omega^{2} S^{2n+1}\rto^-{\ell_{2}}\ddouble 
        & S^{2n-1}\rto\dto^{f} & \anickplus\rto\dto 
        & \Omega S^{2n+1}\ddouble \\ 
     \Omega^{2} S^{2n+1}\rto^-{\ell_{1}} & Y\rto\dto^{g} 
        & X\rto^-{q_{X}}\dto & \Omega S^{2n+1} \\ 
     & \Omega W_{np}\rdouble & \Omega W_{np} & 
  \enddiagram 
\end{equation}  
which defines the space \anickplus\ and the map $\ell_{2}$. 

\medskip\noindent 
\textit{Proof of Theorem~\ref{main}}: 
The top row in~(\ref{Tpb}) is the asserted homotopy fibration, 
which exists if $r+1\geq 3$. Relabel $\ell_{2}$ as $\varphi_{r}$. 
All that is left to check is that $\varphi_{r}$ has the property 
that $E^{2}\circ\varphi_{r}$ is homotopic to the $2^{r+1}$-power 
map on $\Omega^{2} S^{2n+1}$. This follows from juxtaposing the 
squares of connecting maps in~(\ref{l1dgrm}) and~(\ref{Tpb}), 
giving a homotopy commutative diagram 
\[\diagram 
     \Omega^{2} S^{2n+1}\rto^-{\varphi_{r}}\ddouble & S^{2n-1}\dto^{f} \\ 
     \Omega^{2} S^{2n+1}\rto^-{\ell_{1}}\ddouble & Y\dto^{t} \\ 
     \Omega^{2} S^{2n+1}\rto^-{2^{r+1}} & \Omega^{2} S^{2n+1},   
  \enddiagram\] 
and noting that by Lemma~\ref{Yfib} the composite $t\circ f$ is 
homotopic to $E^{2}$. 
$\qqed$

\section{Deferred proofs} 
\label{sec:deferredproofs} 

Two stumbling blocks arise when adapting the odd primary extension 
of $S$ through $\partial_{E}$ in~\cite[Prop 5.1]{T} to the $2$-primary 
case. One is minor and one is major. 

The minor stumbling block is that \curly\ is almost never an 
$H$-space while in the odd primary case $S^{2n+1}\{p^{r}\}$ is 
always an $H$-space. This was used to show that certain Samelson 
products on $\Omega P^{2n+1}(p^{r})$ compose trivially into 
$\Omega S^{2n+1}\{p^{r}\}$. To recover this in the $2$-primary case 
we have to argue differently. In general, suppose we are given maps 
\(a\colon\namedright{P^{s}(2^{r})}{}{\Omega X}\) 
and 
\(b\colon\namedright{P^{t}(2^{r})}{}{\Omega X}\). 
If $r\geq 2$ then the smash $P^{s}(2^{r})\wedge P^{t}(2^{r})$ 
is homotopy equivalent to the wedge $P^{s+t}(2^{r})\vee P^{s+t-1}(2^{r})$.  
This lets us form the mod-$2^{r}$ Samelson product of $a$ and $b$ 
as the composite 
\[\langle a,b\rangle\colon\nameddright{P^{s+t}(2^{r})}{} 
   {P^{s}(2^{r})\wedge P^{t}(2^{r})}{[a,b]}{\Omega X}\] 
where $[a,b]$ is the ordinary Samelson product of $a$ and $b$. 
In our case, let 
\(\nu\colon\namedright{P^{2n}(2^{r})}{}{\Omega\moore}\) 
be the adjoint of the identity map and let 
\(\mu\colon\namedright{S^{2n-1}}{}{\Omega\moore}\) 
be the inclusion of the bottom cell. Let $\mathfrak{ad}^{0}=\mu$, 
$\mathfrak{ad}^{1}$ be the ordinary Samelson product $[\nu,\mu]$, 
and for $j>1$ let  
\(\mathfrak{ad}^{j-1}:\namedright{P^{2nj-1}(2^{r})}{}{\Omega\moore}\) 
be the mod-$p^{r}$ Samelson product defined inductively by 
$\mathfrak{ad}^{j-1}=\langle\nu,\mathfrak{ad}^{j-2}\rangle$. 

\begin{lemma} 
   \label{2rSamelson} 
   Let $r\geq 2$. For $j>1$ the mod-$2^{r}$ Samelson products 
   \(\namedright{P^{2nj-1}(2^{r})}{\mathfrak{ad}^{j-1}}{\Omega\moore}\) 
   all compose trivially with the map 
   \(\namedright{\Omega\moore}{\Omega i}{\Omega\curly}\). 
\end{lemma} 

\begin{proof} 
Since $\Omega i$ is a loop map we have 
$\Omega i\circ\mathfrak{ad}^{k-1}\simeq 
  \langle\Omega i\circ\nu,\Omega i\circ\mathfrak{ad}^{k-2}\rangle$. 
Thus if $\Omega i\circ\mathfrak{ad}^{1}$ is null homotopic then 
iteratively so is $\Omega i\circ\mathfrak{ad}^{k-1}$ for each $k>2$. 
To show $\Omega i\circ\mathfrak{ad}^{1}$ is null homotopic it is 
equivalent to adjoint and show that the mod-$2^{r}$ Whitehead product 
\(\namedright{P^{4n}(2^{r})}{\widetilde{\mathfrak{ad}}^{1}}{\moore}\) 
composes trivially with 
\(\namedright{\moore}{i}{\curly}\). 

From the homotopy cofibration 
\(\nameddright{S^{2n+1}}{\underline{2}^{r}}{S^{2n+1}}{}{P^{2n+1}(2^{r})}\) 
we obtain a homotopy fibration diagram 
\[\diagram 
     \Omega S^{2n+1}\rto\dto & \curly\rto\dto^{\gamma} 
         & S^{2n+1}\rto^-{\underline{2}^{r}}\dto & S^{2n+1}\dto \\ 
     \Omega P^{2n+1}(s^{r})\rdouble & \Omega P^{2n+1}(2^{r})\rto 
         & \ast\rto & P^{2n+1}(2^{r}) 
  \enddiagram\] 
which defines the map $\gamma$. The left square implies 
that $\gamma_{\ast}$ is degree one in $H_{2n}(\ )$, and so the Bockstein 
implies that $\gamma_{\ast}$ is the identity on the bottom Moore space. 
That is, the composite 
\(\nameddright{\moore}{i}{\curly}{\gamma}{\Omega P^{2n+1}(2^{r})}\) 
is the inclusion of the bottom Moore space and so it is homotopic 
to the suspension $E$. Any Whitehead product suspends trivially, so 
$E\circ\widetilde{\mathfrak{ad}}^{1}$ is null homotopic. We claim 
that $\gamma$ is a homotopy equivalence on $4n$-skeletons, in which 
case $i\circ\widetilde{\mathfrak{ad}}^{1}$ is null homotopic and we 
are done. 

It suffices to show that $\gamma_{\ast}$ is an isomorphism in mod-$2$ 
homology when restricted to $4n$-skeletons. The Serre spectral sequence 
for the homotopy fibration 
\(\nameddright{\Omega S^{2n+1}}{}{\curly}{}{S^{2n+1}}\)  
collapses, implying that $\hlgy{\curly}\cong\zmodtwo[x]\otimes\Lambda(y)$, 
where the degrees of $x$ and $y$ are $2n$ and $2n+1$ respectively. 
A basis for the homology of the $4n$-skeleton of \curly\ is 
therefore $\{x,y,x^{2}\}$. On the other hand, since $P^{2n+2}(2^{r})$ 
is a suspension, the Bott-Samelson Theorem implies that 
$\hlgy{\Omega P^{2n+1}(2^{r})}\cong T(u,v)$, 
where $T(\ )$ is the free tensor algebra, and $u$ and $v$ have 
degrees $2n$ and $2n+1$ respectively. A basis for the homology of 
the $4n$-skeleton of $\Omega P^{2n+2}(2^{r})$ is therefore 
$\{u,v,u^{2}\}$. Since $\gamma$ is the identity on the bottom 
Moore space, we have $\gamma_{\ast}(x)=u$ and $\gamma_{\ast}(y)=v$. 
The left square in the fibration diagram above also implies that 
$\gamma_{\ast}(x^{2})=u^{2}$. Thus $\gamma_{\ast}$ is an isomorphism 
on the $4n$-skeletons.  
\end{proof} 

The major stumbling block is that to get the inductive argument 
in~\cite[2.1]{T} off the ground requires the existence of a homotopy 
commutative square 
\begin{equation} 
  \label{majorobstr} 
  \diagram 
     S^{2n-1}\times\Omega^{2} S^{2n+1}\rto^-{E^{2}\cdot 1}\dto^{\pi_{2}} 
       & \Omega^{2} S^{2n+1}\dto^{S} \\ 
     \Omega^{2} S^{2n+1}\rto^-{S} & \Omega^{2} S^{4n+1}\{2\}.   
  \enddiagram 
\end{equation}  
The odd primary version, with $S$ mapping to $\Omega^{2} S^{2np+1}\{p\}$, 
does commute because it is known that $S$ is an $H$-map. The $2$-primary 
version in~(\ref{majorobstr}) does not commute because it does not even 
commute in homology. To see this, observe that if $x\in H_{2n-1}(S^{2n-1})$ 
and $a\in H_{2n-1}(\Omega^{2} S^{2n+1})$ are generators, then 
$(E^{2}\cdot 1)_{\ast}(x,a)=a^{2}$. Note that $a^{2}$ is the Bockstein 
of the generator $b\in H_{4n-1}(\Omega^{2} S^{2n+1})$. Under $\Omega H$ 
we have $(\Omega H)_{\ast}(b)=v$, where $v\in H_{4n-1}(\Omega^{2} S^{4n+1})$ 
is a generator. Since $S$ is a lift of $\Omega H$ through 
\(\namedright{\Omega^{2} S^{4n+1}\{2\}}{\Omega\delta}{\Omega^{2} S^{2n+1}}\) 
and $(\Omega\delta)_{\ast}$ is onto, we have $S_{\ast}(b)=\overline{v}$, 
where $(\Omega\delta)_{\ast}(\overline{v})=v$. Now $\beta\overline{v}$ 
is nonzero and generates $H_{4n-2}(\Omega^{2} S^{4n+1}\{2\})$. Thus 
$S_{\ast}(a^{2})=\beta\overline{v}$ and so 
$(S\circ(E^{2}\cdot 1))_{\ast}(x,y)\neq 0$. On the other hand, if the 
diagram above were to commute then we would have 
$(S\circ(E^{2}\cdot 1))_{\ast}(x,y)=0$. 

To get around this obstruction, we introduce a degree $2$ map on $S^{2n-1}$. 
In Lemma~\ref{k=1case} we show that 
$S\circ((E^{2}\circ\underline{2})\cdot 1)$ does factor as $S\circ\pi_{1}$, 
and then we use this as the base case in a modified induction to 
prove Proposition~\ref{Ebarextension}. 

To prepare for Lemma~\ref{k=1case}, we make three preliminary 
observations. First, in general, if $A$ 
and~$B$ are spaces, then there is a homotopy equivalence 
\(e\colon\namedright{\Sigma A\vee\Sigma B\vee(\Sigma A\wedge B)}{\simeq} 
    {\Sigma(A\times B)}\) 
given by $\Sigma i_{1}+\Sigma i_{2}+j$, where $i_{1}$ and $i_{2}$ are 
the inclusions of the left and right factors respectively, and~$j$ is the 
canonical map of a join into the suspension of a product. If $A=B$ 
and $B=\Omega X$ with loop multiplication $\mu$ then the composite 
\(\mu^{\ast}\colon\nameddright{\Sigma\Omega X\wedge\Omega X}{j} 
    {\Sigma(\Omega X\times\Omega X)}{\Sigma\mu}{\Sigma\Omega X}\) 
is the Hopf construction on $\Omega X$ and Ganea~\cite{Ga} showed that 
there is a homotopy fibration 
\(\nameddright{\Sigma\Omega X\wedge\Sigma\Omega X}{\mu^{\ast}} 
     {\Sigma\Omega X}{ev}{X}\). 

The second observation is a well-known factorization of the
$2^{nd}$-power map on $\Omega^{2} S^{2n+1}\{2\}$. Consider the homotopy
cofibration sequence
\(\namedddright{S^{2}}{\underline{2}}{S^{2}}{i}{P^{3}(2)}
   {q}{S^{3}}\),
where $P^{3}(2)$ is the mod-$2$ Moore space, $i$ is the inclusion of
the bottom cell, and $q$ is the pinch onto the top cell. This cofibration
sequence induces a homotopy fibration sequence
\[\namedddright{\mbox{Map}_{\ast}(S^{3},S^{2n+1})}{q^{\ast}}
    {\mbox{Map}_{\ast}(P^{3}(2),S^{2n+1})}{i^{\ast}}
    {\mbox{Map}_{\ast}(S^{2},S^{2n+1})}{\underline{2}^{\ast}}
    {\mbox{Map}_{\ast}(S^{2},S^{2n+1})}.\]
Since $\mbox{Map}_{\ast}(S^{2},S^{2n+1})=\Omega^{2} S^{2n+1}$ and
$\underline{2}^{\ast}$ is the $2^{nd}$-power map on $\Omega^{2} S^{2n+1}$,
this homotopy fibration sequence can be identified with the homotopy
fibration sequence
\[\namedddright{\Omega^{3} S^{2n+1}}{\Omega\partial}{\Omega^{2} S^{2n+1}\{2\}}
    {\Omega\delta}{\Omega^{2} S^{2n+1}}{2}{\Omega^{2} S^{2n+1}}.\]
In particular, $\mbox{Map}_{\ast}(P^{3}(2),S^{2n+1})=\Omega^{2} S^{2n+1}\{2\}$,
and so the $2^{nd}$-power map on $\Omega^{2} S^{2n+1}\{2\}$ is determined
by the degree~$2$ map on $P^{3}(2)$. This factors as the composite
\(\namedddright{P^{3}(2)}{q}{S^{3}}{\eta}{S^{2}}{i}{P^{3}(2)}\),
where $\eta$ is an element of Hopf invariant one. Thus we have the 
following.

\begin{lemma}
   \label{curly2power}
   There is a homotopy commutative diagram
   \[\diagram
        \Omega^{2} S^{2n+1}\{2\}\rrto^-{2}\dto^{\Omega\delta}
            & & \Omega^{2} S^{2n+1}\{2\}\ddouble \\
        \Omega^{2} S^{2n+1}\rto^-{\eta^{\ast}}
            & \Omega^{3} S^{2n+1}\rto^-{\Omega\partial}
            & \Omega^{2} S^{2n+1}\{2\}.
     \enddiagram\]
   $\qqed$
\end{lemma}

The third observation is a consequence of the James construction. 
Given a map 
\(f\colon:\namedright{A}{}{\Omega X}\), 
the James construction implies that there is a unique $H$-map 
\(\overline{f}\colon\namedright{\Omega\Sigma A}{}{\Omega X}\) 
such that $\overline{f}\circ E\simeq f$. The uniqueness assertion 
immediately implies the following lemma. 

\begin{lemma} 
   \label{Jamescase} 
   Suppose $A=\Omega A^{\prime}$ and $f$ is an $H$-map. Then there 
   is a homotopy commutative diagram 
   \[\diagram 
        \Omega\Sigma\Omega A^{\prime}\rto^-{\overline{f}}\dto^{\Omega ev} 
             & \Omega X\ddouble \\ 
        \Omega A^{\prime}\rto^-{f} & \Omega X. 
     \enddiagram\] 
   $\qqed$ 
\end{lemma} 

Now we are ready to prove the modified square. 

\begin{lemma} 
   \label{k=1case} 
   There is a homotopy commutative diagram 
   \[\diagram 
        S^{2n-1}\times\Omega^{2} S^{2n+1}
             \rto^-{(E^{2}\circ\underline{2})\cdot 1}\dto^{\pi_{2}} 
          & \Omega^{2} S^{2n+1}\dto^{S} \\ 
        \Omega^{2} S^{2n+1}\rto^-{S} & \Omega^{2} S^{4n+1}\{2\}.   
     \enddiagram\] 
\end{lemma} 

\begin{proof} 
The adjoint of the upper direction around the diagram is homotopic 
to the composite 
\(\psi\colon\lllnameddright{\Sigma(S^{2n-1}\times\Omega^{2} S^{2n+1})} 
    {\Sigma(E^{2}\circ\underline{2})\cdot 1} 
    {\Sigma\Omega^{2} S^{2n+1}}{\widetilde{S}}{\Omega S^{4n+1}\{2\}}\) 
where $\widetilde{S}$ is the adjoint of $S$. Precomposing with the 
homotopy equivalence  
\(e\colon\namedright{\Sigma S^{2n-1}\vee\Sigma\Omega^{2} S^{2n+1} 
    \vee(\Sigma S^{2n-1}\wedge\Omega^{2} S^{2n+1})}{\simeq} 
    {\Sigma(S^{2n-1}\times\Omega^{2} S^{2n+1})}\), 
the assertion of the lemma is therefore equivalent to saying that 
the restriction of $\psi\circ e$ to the first and third wedge summands 
is null homotopic. As the restriction to the first wedge summand is 
null homotopic by connectivity, it remains only to show that $\psi\circ j$ 
is null homotopic. 

Consider the map $\Sigma((E^{2}\circ\underline{2})\cdot 1)$. 
By definition, 
$(E^{2}\circ\underline{2})\cdot 1\simeq 
   \mu\circ ((E^{2}\circ\underline{2})\times 1)$, 
where $\mu$ is the loop space multiplication on $\Omega^{2} S^{4n+1}\{2\}$.  
The naturality of $j$ therefore implies that $\psi\circ j$ is homotopic to 
$\widetilde{S}\circ\Sigma\mu\circ j\circ 
    (\Sigma (E^{2}\circ\underline{2})\wedge 1)$, 
that is, to 
$\widetilde{S}\circ\mu^{\ast}\circ 
    (\Sigma (E^{2}\circ\underline{2})\wedge 1)$. 
Now 
$\Sigma (E^{2}\circ\underline{2})\wedge 1\simeq  
    (\Sigma E^{2}\wedge 1)\circ (\Sigma\underline{2}\wedge 1)$, 
and $\Sigma\underline{2}\wedge 1$ is homotopic to the map of degree $2$ 
on $\Sigma S^{2n-1}\wedge\Omega^{2} S^{2n+1}$. Since this is the suspension 
of the degree $2$ map on $S^{2n-1}\wedge\Omega^{2} S^{2n+1}$, 
the adjoint of $\psi\circ j$ is homotopic to the composite 
\(\nameddright{S^{2n-1}\wedge\Omega^{2} S^{2n+1}}{\underline{2}} 
     {S^{2n-1}\wedge\Omega^{2} S^{2n+1}}{\psi^{\prime}} 
     {\Omega^{2} S^{4n+1}\{2\}}\) 
where $\psi^{\prime}$ is the adjoint of 
$\widetilde{S}\circ\mu^{\ast}\circ(\Sigma E^{2}\wedge 1)$. As 
$\Omega^{2} S^{4n+1}\{2\}$ is a loop space, 
$\psi^{\prime}\circ\underline{2}$ is homotopic to the composite 
\(\nameddright{S^{2n-1}\wedge\Omega^{2} S^{2n+1}}{\psi^{\prime}} 
    {\Omega^{2} S^{4n+1}\{2\}}{2}{\Omega^{2} S^{4n+1}\{2\}}\). 
 
By Lemma~\ref{curly2power}, the $2^{nd}$-power map on 
$\Omega^{2} S^{4n+1}\{4\}$ is homotopic to the composite 
\(\namedddright{\Omega^{2} S^{4n+1}\{2\}}{\Omega\delta}{\Omega^{2} S^{4n+1}} 
     {\eta^{\ast}}{\Omega^{3} S^{4n+1}}{}{\Omega^{2} S^{4n+1}\{2\}}\). 
Thus $2\circ\psi^{\prime}$ factors through 
$\Omega\delta\circ\psi^{\prime}$. Since $\psi^{\prime}$ is the adjoint 
of $\psi$, it is homotopic to the composite $\Omega\psi\circ E$, which 
by the definition of $\psi$, is the composite 
$\Omega\widetilde{S}\circ\Omega\mu^{\ast}\circ 
   \Omega(\Sigma E^{2}\wedge 1)\circ E$. 
We claim that $\Omega\delta\circ\Omega\widetilde{S}\circ\Omega\mu^{\ast}$ 
is null homotopic, in which case $\Omega\delta\circ\psi^{\prime}$ is 
null homotopic and we are done. To see this, observe that the composite 
\(\nameddright{\Omega^{2} S^{2n+1}}{E}{\Omega\Sigma\Omega^{2} S^{2n+1}} 
     {\Omega\widetilde{S}}{\Omega^{2} S^{4n+1}\{2\}}\) 
is homotopic to $S$ since $\widetilde{S}$ is the adjoint of $S$. Thus 
$\Omega\delta\circ\Omega\widetilde{S}$ is an $H$-map such that 
$(\Omega\delta\circ\Omega\widetilde{S})\circ E\simeq\Omega H$. 
By Lemma~\ref{Jamescase}, this implies that there is a homotopy 
commutative square 
\[\diagram 
     \Omega\Sigma\Omega^{2} S^{2n+1}
            \rto^-{\Omega\widetilde{S}}\dto^{\Omega ev} 
       & \Omega^{2} S^{4n+1}\{2\}\dto^{\Omega\delta} \\ 
     \Omega^{2} S^{2n+1}\rto^-{\Omega H} & \Omega^{2} S^{4n+1}. 
  \enddiagram\] 
Thus there is a string of homotopies   
\[\Omega\delta\circ\Omega\widetilde{S}\circ\Omega\mu^{\ast}\simeq 
       \Omega H\circ\Omega ev\circ\Omega\mu^{\ast}\simeq 
       \Omega H\circ\ast\simeq\ast,\] 
where Ganea's fibration has been used to show that 
$ev\circ\mu^{\ast}\simeq\ast$. This completes the proof. 
\end{proof} 

With the stumbling blocks addressed, we move on to prove  
Propositions~\ref{Ebarextension} and~\ref{Xextension}.  

\noindent 
\textit{Proof of Proposition~\ref{Ebarextension}}: 
We follow the proof of the odd primary analogue in~\cite[2.1]{T}. 
The idea is to filter $\overline{E}$ by certain homotopy pushouts 
and construct extensions iteratively as pushout maps. The case of $E$ 
is carried along to indicate differences, and where modifications 
arise. 

The mod-$2$ homology Serre spectral sequence for the homotopy fibration 
\(\nameddright{\Omega S^{2n+1}}{}{F}{}{P^{2n+1}(p^{r})}\) 
shows that $\rhlgy{F;\zmodtwo}$ is generated as a vector space by 
elements $x_{2nk}$ in degrees $2nk$ for $k\geq 1$. Let $F_{k}$   
be the $2nk$-skeleton of $F$ and observe that there is a homotopy 
cofibration 
\(\nameddright{S^{2nk-1}}{g_{k-1}}{F_{k-1}}{}{F_{k}}\). 
Define corresponding filtrations on $E$ and $\overline{E}$ by homotopy 
pullback diagrams   
\begin{equation} 
  \label{EkFk} 
  \diagram 
      E_{k-1}\rto\dto & E\dto 
         & & \overline{E}_{k-1}\rto\dto & \overline{E}\dto \\ 
      F_{k-1}\rto\dto & F\dto^{\imath} 
         & & F_{k-1}\rto\dto & F\dto^{\overline{\imath}} \\ 
      \Omega S^{2n+1}\rdouble & \Omega S^{2n+1} 
         & & \Omega S^{2n+1}\rdouble & \Omega S^{2n+1}. 
  \enddiagram 
\end{equation} 
Note that $F_{0}=\ast$ so $E_{0}\simeq\Omega^{2} S^{2n+1}$ and 
$\overline{E}_{0}\simeq\Omega^{2} S^{2n+1}$. As $F=\varinjlim F_{k-1}$ 
we have $E=\varinjlim E_{k-1}$ and 
$\overline{E}=\varinjlim\overline{E}_{k-1}$. 

The filtration of $E$ implies by~\cite[4.2]{T} or~\cite[2.2]{GT} that 
for $k\geq 1$ there are homotopy pushouts 
\[\diagram 
      S^{2nk-1}\times\Omega^{2} S^{2n+1}
              \rto^-{\theta_{k-1}}\dto^{\pi_{2}} 
         & E_{k-1}\dto \\ 
      \Omega^{2} S^{2n+1}\rto & E_{k}  
  \enddiagram\] 
where $\pi_{2}$ is the projection, $\theta_{0}\simeq E^{2}\cdot 1$, 
and for $k>1$, $\theta_{k-1}$ restricted to $S^{2nk-1}$ is a lift of 
the map 
\(\namedright{S^{2nk-1}}{g_{k-1}}{F_{k-1}}\). 
By Lemma~\ref{2rSamelson}, the argument of~\cite[5.1]{T} goes through 
without change for $p=2$ and $r>1$. The outcome is that the lift of $g_{k-1}$ 
can be chosen to be divisible by $2^{r}$, and this is used to show 
that $\theta_{k-1}$ factors as a composite 
\(\lnameddright{S^{2nk-1}\times\Omega^{2} S^{2n+1}} 
      {\underline{2}^{r}\times 1}{S^{2nk-1}\times\Omega^{2} S^{2n+1}} 
      {\psi_{k-1}}{\overline{E}_{k-1}}\) 
for some map $\psi_{k-1}$. We would like to extend the map  
\(\namedright{\Omega^{2} S^{2n+1}}{S}{\Omega^{2} S^{4n+1}\{2\}}\) 
across each $E_{k}$ by a pushout map, but the sticking point 
is $\theta_{0}$, as~(\ref{majorobstr}) implies that it is not 
possible to extend $S$ across $E_{1}$.  

Instead, we replace $E$ by $\overline{E}$. Now the filtration 
of $\overline{E}$ implies by~\cite[4.2]{T} or~\cite[2.2]{GT} that 
for $k\geq 1$ there are homotopy pushouts 
\[\diagram 
      S^{2nk-1}\times\Omega^{2} S^{2n+1}
              \rto^-{\overline{\theta}_{k-1}}\dto^{\pi_{2}} 
         & \overline{E}_{k-1}\dto \\ 
      \Omega^{2} S^{2n+1}\rto & \overline{E}_{k}  
  \enddiagram\] 
where $\pi_{2}$ is the projection, 
$\overline{\theta}_{0}\simeq(E^{2}\circ\underline{2})\cdot 1$,  
and for $k>1$, $\overline{\theta}_{k-1}$ restricted to $S^{2nk-1}$ is a 
lift of the map 
\(\namedright{S^{2nk-1}}{g_{k-1}}{F_{k-1}}\). 
The presence of $\underline{2}$ in $\overline{\theta}_{0}$ is due 
to $\overline{\imath}$ being of degree~$2$ on the bottom cell, 
whereas $\imath$ is degree one. Also, since the map  
\(\namedright{E}{}{F}\) 
factors as the composite 
\(\nameddright{E}{}{\overline{E}}{}{F}\), 
the same holds true at each filtration level, so the lift of $g_{k-1}$  
to $\overline{E}_{k-1}$ can be chosen to be the composite 
\(\nameddright{S^{2nk-1}}{}{E_{k-1}}{}{\overline{E}_{k-1}}\), 
where the left map is divisible by $2^{r}$. This can then be 
used to show argue as before that for $k>1$, 
$\overline{\theta}_{k-1}$ factors as a composite 
\(\lnameddright{S^{2nk-1}\times\Omega^{2} S^{2n+1}} 
      {\underline{2}^{r}\times 1}{S^{2nk-1}\times\Omega^{2} S^{2n+1}} 
      {\overline{\psi}_{k-1}}{\overline{E}_{k-1}}\) 
for some map $\overline{\psi}_{k-1}$. 

When $k=1$, Lemma~\ref{k=1case} implies that there is a pushout map 
\(\epsilon_{1}\colon\namedright{\overline{E}_{1}}{} 
    {\Omega^{2} S^{4n+1}\{2\}}\) 
which extends $S$. When $k>1$, the factorization 
of $\overline{\theta}_{k-1}$ through $\underline{2}^{r}\times 1$ 
where $r\geq 2$, together with the fact that $\Omega^{2} S^{4n+1}\{2\}$ 
is a homotopy associative $H$-space whose $4^{th}$-power map 
is null homotopic~\cite{N}, lets us apply~\cite[4.3]{T} or~\cite[2.3]{GT} 
to show that for each $k>1$ there is a pushout map  
\(\epsilon_{k}\colon\namedright{\overline{E}_{k}}{}  
    {\Omega^{2} S^{4n+1}\{2\}}\) 
which extends $\epsilon_{k-1}$. In the limit we therefore obtain a map 
\(\namedright{\overline{E}}{e_{1}}{\Omega^{2} S^{4n+1}\{2\}}\) 
which extends $S$, as asserted. 
$\qqed$ 
\medskip 

\noindent 
\textit{Proof of Proposition~\ref{Xextension}}: 
This follows the proof of the odd primary analogue in~\cite[2.2]{T}. 
Filter $X$ by homotopy pullbacks 
\[\diagram 
     Y\rdouble\dto & Y\dto \\ 
     X_{k-1}\rto\dto & X\dto^-{q_{X}} \\ 
     \jk\rto & \Omega S^{2n+1} 
  \enddiagram\] 
where \jkm\ is the $(2nk-1)$-skeleton of $\Omega S^{2n+1}$, 
and note that there is a homotopy cofibration 
\(\nameddright{S^{2nk-1}}{}{\jkm}{}{\jk}\). 
Note that $J_{0}(S^{2n})\simeq\ast$ so $X_{0}\simeq Y$, and as 
$\Omega S^{2n+1}=\varinjlim \jk$, we have $X=\varinjlim X_{k}$.  
Since~\cite[6.1]{T} holds without change for $p=2$ and $r>1$, the argument  
in~\cite[7.3]{T} goes through to show that there are homotopy pushouts 
\[\diagram 
     S^{2nk-1}\times Y\rto^-{\theta^{\prime}_{k-1}}\dto^{\pi_{2}}  
        & X_{k-1}\dto \\ 
     Y\rto & X_{k} 
  \enddiagram\] 
where $\pi_{2}$ is the projection and $\Sigma\theta^{\prime}_{k-1}$ 
restricted to $\Sigma(S^{2nk-1}\rtimes\Omega^{2} S^{2n+1})$ 
is divisible by $2^{r}$ for each $k\geq 1$. Now start with the map 
\(\epsilon^{\prime}_{0}\colon\namedright{X_{0}=Y}{g}{\Omega W_{2n}}\). 
The divisibility property of $\Sigma\theta^{\prime}_{k-1}$, 
together with the fact from Theorem~\ref{Wnpower} that $\Omega W_{2n}$ 
is a homotopy associative $H$-space whose $4^{th}$-power map 
is null homotopic, lets us apply~\cite[4.3]{T} or~\cite[2.3]{GT} 
to show that for each $k\geq 1$ there is a pushout map 
\(\epsilon^{\prime}_{k}\colon\namedright{X_{k}}{}{\Omega W_{2n}}\) 
which extends $\epsilon^{\prime}_{k-1}$, and therefore extends $g$, 
as asserted.  
$\qqed$

\bibliographystyle{amsalpha}

\end{document}